\documentclass[12pt, a4paper]{amsart}

\usepackage[hmargin=30mm, vmargin=25mm, includefoot, twoside]{geometry}
\usepackage[bookmarksopen=true]{hyperref}

\usepackage{amsfonts,amssymb,verbatim}
\usepackage{latexsym}
\usepackage{mathrsfs}
\usepackage{stmaryrd}
\usepackage{xspace}
\usepackage{enumerate, paralist}
\usepackage{graphicx}
\usepackage[all]{xy}
\usepackage{extarrows}

\usepackage[usenames,dvipsnames]{color}

\usepackage{txfonts, pxfonts}

\usepackage{amsthm}
\usepackage{amsmath}

\newtheorem{thm}{Theorem}[section]
 
 \newtheorem{lem}[thm]{Lemma}
 \newtheorem{prop}[thm]{Proposition}
  
  \newtheorem*{thmA}{Theorem A}

\numberwithin{equation}{section}

 \theoremstyle{definition}
  \newtheorem{defn}[thm]{Definition}
 \theoremstyle{remark}
 \newtheorem{rem}[thm]{Remark}
  \newtheorem{ex}[thm]{Example}

\def\X{\mathfrak{X}}
\def\Xl{\hat{\mathfrak{X}}}
\def\BXXl{\mathfrak{B}(\mathfrak{X},\hat{\mathfrak{X}}^*)}
\def\BX{\mathfrak{B}(\mathfrak{X})}
\def\KX{\mathfrak{K}(\mathfrak{X})}
\def\BXl{\mathfrak{B}(\hat{\mathfrak{X}}^*)}

\def\LE{L^1(X;E)}

\def\LpE{L^p(X;E)}
\def\BLpE{\mathfrak{B}(L^p(X;E))}
\def\LqEst{L^q(X;E^*)}
\def\Bd{\mathfrak{B}(L^p(X;E),L^q(X;E^*)^*)}

\def\d{\mathrm{d}}
\def\supp{\mathrm {supp}}
\def\andx{\quad\mbox{~and~}\quad}

\def\Roe{\mathrm{Roe}(X,B)}
\def\K{\mathcal{K}(X,B)}

\def\VL{\mathrm{VL}(X)}
\def\VLinf{\mathrm{VL}_\infty(X)}

\def\N{\mathcal N}
\def\Xs{\mathcal X}
\def\Ys{\mathcal Y}

\def\E{\mathcal E}
\def\G{\mathfrak{G}}
\def\l{\langle}
\def\r{\rangle}
\def\W*OT{\mathrm{W^*OT}}

\def\Commut{\mathrm{Commut}}

\begin{document}

\title[A quasi-local characterisation of $L^p$-Roe algebras]{A quasi-local characterisation of $L^p$-Roe algebras}

\author{Kang Li$^{1}$}
\address{Institute of Mathematics of the Polish Academy of Sciences, ul. \'{S}niadeckich 8, 00-656 Warsaw, Poland.}
\email{kli@impan.pl}

\author{Zhijie Wang$^{2}$}
\address{College of Mathematics Physics and Information Engineering, Jiaxing University,
Yuexiu Road (South) 56, 314001 Zhejiang, China.}

\email{wangzhijie112@gmail.com}
\author{Jiawen Zhang$^{3}$}
\address{School of Mathematics, University of Southampton, Highfield SO17 1BJ, United Kingdom.}
\email{jiawen.zhang@soton.ac.uk}

\thanks{{$^{1}$} Supported by the Danish Council for Independent Research (DFF-5051-00037), Deutsche Forschungsgemeinschaft (SFB 878, Groups, Geometry and Actions) and  the European Research Council (ERC-677120)}

\thanks{{$^{2}$} Supported by NSFC (No. 11501249)}
\thanks{{$^{3}$} Supported by the Sino-British Trust Fellowship by Royal Society, International Exchanges 2017 Cost Share (China) grant EC$\backslash$NSFC$\backslash$170341, and NSFC11871342.}

%\date{}
%\subjclass[2010]{20F65, 20F67, 20F69}

%\thanks{}
\baselineskip=16pt

\begin{abstract}
Very recently, \v{S}pakula and Tikuisis provide a new characterisation of (uniform) Roe algebras via quasi-locality when the underlying metric spaces have straight finite decomposition complexity. In this paper, we improve their method to deal with the $L^p$-version of (uniform) Roe algebras for any $p\in [1,\infty)$. Due to the lack of reflexivity on $L^1$-spaces, some extra work is required for the case of $p=1$.

\end{abstract}
%\date{\today}
\maketitle

\parskip 4pt


\section{Introduction}
(Uniform) Roe algebras are $C^*$-algebras associated to metric spaces, which reflect coarse properties of the underlying metric spaces. These algebras have been well-studied and have fruitful applications, among which the most important ones would be the (uniform) coarse Baum-Connes conjecture and the Novikov conjecture (e.g., \cite{MR1905840, MR2523336, Yu95, MR1451759, MR1626745, Yu00}). Meanwhile, they also provide a link between coarse geometry of metric spaces and the theory of $C^*$-algebras (e.g., \cite{ALLW17, MR1876896, MR1739727, MR3158244, LL, LW18, MR1763912, MR2873171, Scarparo:2016kl, MR1905840, MR2800923, WZ10}), and turn out to be useful in the study of topological phases of matter (e.g., \cite{Kub,EwertMeyer}) as well as the theory of limit operators in the study of Fredholmness of band-dominated operators (e.g., \cite{MR3451966, MR3212726, MR3151282, vspakula2017metric}).

By definition, the (uniform) Roe algebra of a proper metric space $X$ is the norm closure of all bounded locally compact operators $T$ with \emph{finite propagation} in the sense that there exists $R>0$ such that for any $f,g\in C_b(X)$ acting on $L^2(X)$ by pointwise multiplication, we have $fTg=0$ provided their supports are $R$-separated (i.e., the distance between the supports of $f$ and $g$ is at least $R$). Since general elements in (uniform) Roe algebras may not have finite propagation, it is usually difficult to tell what operators exactly belong to them. On the other hand, Roe \cite{MR1399087} defined an asymptotic version of finite propagation as follows: An operator $T$ on $L^2(X)$ has \emph{finite $\varepsilon$-propagation} for $\varepsilon>0$, if there is $R>0$ such that for any $f,g\in C_b(X)$, we have $\|fTg\|\leq \varepsilon \|f\|\cdot\|g\|$ provided their supports are $R$-separated. Operators with finite $\varepsilon$-propagation for all $\varepsilon>0$ are called \emph{quasi-local} in \cite{MR918459}. It is clear that limits of finite $\varepsilon$-propagation operators still have finite $\varepsilon$-propagation. Consequently, all operators in (uniform) Roe algebras are quasi-local.

A natural question is that whether the converse holds as well, i.e., does every locally compact quasi-local operator belong to the (uniform) Roe algebra? An affirmative answer to this question would provide a new approach to detect what operators belong to these algebras in a more practical way by estimating the norms of operator-blocks far from strips around the diagonal, and it has several immediate consequences including the followings.

The first one has its root in Engel's work \cite[Section~2]{Engel18}, where he studied the index theory of pseudo-differential operators. He showed that the indices of uniform pseudo-differential operators on Riemannian manifolds are quasi-local, while it is unclear to him whether they live in Roe algebras, which are well-understood. Another application is in the work of White and Willett \cite{WW} on Cartan subalgebras of uniform Roe algebras. They showed that if two uniform Roe algebras of bounded geometry metric spaces with Property A are $*$-isomorphic, then the underlying metric spaces are bijectively coarsely equivalent provided that every quasi-local operator belongs to the uniform Roe algebras.

Historically, this question has been studied and partially addressed by many people including Lange and Rabinovich for $X=\mathbb{Z}^n$ \cite{MR790430} (in fact they worked in a more general context, see the next paragraph), Engel for $X$ is a manifold of bounded geometry with polynomial volume growth \cite{Engel15}, \v{S}pakula and Tikuisis \cite{spakula2017relative} for $X$ has straight finite decomposition complexity in the sense of \cite{dranishnikov2014asymptotic}. To our best knowledge, this question is still open for general metric spaces.

Based on the original definitions, various versions of Roe algebras are proposed and studied by different purposes. In fact, in recent years there has been an uptick in interest in the $L^p$-version of (uniform) Roe algebras for $p\in [1,\infty)$, from the communities of both limit operator theory and coarse geometry (e.g. \cite{MR3151282, MR3212726, vspakula2017metric, MR3451966, chung2018rigidity, Zhang18}). And it is natural and important to study the same question in this context, i.e., does every locally compact and quasi-local operator belong to the $L^p$-version of (uniform) Roe algebras for $p\in [1,\infty)$?

In this paper, we improve the method of \v{S}pakula and Tikuisis \cite{spakula2017relative} in order to generalise their result from the case of $p=2$ to any $p\in [1,\infty)$. The main part of our result is the following (see Theorem~\ref{char for Roe alg Thm} for the complete version), which answers the $L^p$-version of the question above under the condition that the underlying metric space has straight finite decomposition complexity. 

\begin{thmA}
For a proper metric space with straight finite decomposition complexity and $p \in [1,+\infty)$, quasi-locality is equivalent to being in the associated $L^p$-Roe-like algebra.
\end{thmA}

Here the notion of \emph{$L^p$-Roe-like algebra} is the $L^p$-analogue of Roe-like algebras \v{S}pakula and Tikuisis introduced for $p=2$ (\cite[Definition~2.3]{spakula2017relative}) and for which their main result is established. However, we would like to point out that our definition of $L^p$-Roe-like algebras are more general than \v{S}pakula and Tikuisis' definition even for $p=2$, as we drop a commutant condition in \cite[Definition~2.3]{spakula2017relative}, which is used in the proof of their main theorem. However, we observe that this condition is redundant for the proof of the main theorem if we replace it with Lemma~\ref{2.2 replacement} below. The reason we drop this condition is inspired by the fact that it is not fulfilled for general $L^1$-Roe-like algebras, and an obvious advantage of doing this is to allow more examples especially in the case of $p=1$ (see Remark~\ref{weak def rem} and  Example~\ref{uniform Roe alg} for more details).

The proof of our main theorem is closely modelled on their original one in \cite{spakula2017relative} at least for $p\in (1,\infty)$, except that the $L^p$-Roe-like algebras need not possess a bounded involution and von Neumann algebra techniques are invalid. Instead, we have to deal with asymmetric situation as in the proof of the implication ``(iii) $\Rightarrow$ (i)" in Theorem~\ref{char for Roe alg Thm} and provide a direct and concrete proof of Lemma~\ref{construction}.

The case of $p=1$ is more complicated and in this case Proposition~\ref{estimate for cmmt cor} is established, which is the most technical part of the paper and is also a generalisation of \cite[Corollary~4.3]{spakula2017relative}. The difficulty comes from the lack of reflexivity on $L^1$-spaces, and the trick of the proof is to consider an artificial space $L^0(X)$, which lies between $C_0(X)$ and $L^\infty(X)$. It is worth pointing out that Proposition~\ref{estimate for cmmt cor} is based on a crucial intermediate result established in a more general setup of Banach spaces, and we hope that there might be some other applications in the future.\footnote{After we finish this paper, \v{S}pakula and the third-named author informed us that the main theorem of this paper remains true if we only require Property A rather than straight finite decomposition complexity \cite{vspakula2018quasi}. Their arguments include an essential application of Proposition~\ref{estimate for cmmt cor}.}

The paper is organised as follows: we establish the settings of the paper by recalling some background  in Banach algebra theory and coarse geometry in Section 2, where various examples of $L^p$-Roe-like algebras are also provided. In Section 3, we provide a complete version of our main result Theorem A, and prove the relatively easier part, where the assumption of straight finite decomposition complexity is not required. In Section 4, we prove the technical tool, Proposition~\ref{estimate for cmmt cor}, and finish the remaining proof of the main theorem.

\textbf{Conventions:} Let $\X$ be a Banach space. We denote the closed unit ball of $\X$ by $\X_1$. For any $a,b \in \X$ and $\varepsilon >0$, we denote $\|a-b\| \leq \varepsilon$ by $a\approx_\varepsilon b$. We also denote the bounded linear operators on $\X$ by $\BX$, and the compact operators on $\X$ by $\KX$. Moreover, for a Banach algebra $A$ we define
$$A_\infty := \ell^\infty(\mathbb N, A)\big/ \big\{(a_n)_{n\in \mathbb N} \in \ell^\infty(\mathbb N, A): \lim_{n\rightarrow \infty}\|a_n\|=0\big\},$$
which is a Banach algebra with respect to the quotient norm.

\emph{Throughout the paper, we fix a \emph{proper} metric space $(X,d)$} (i.e., every bounded subset is pre-compact). Note that such a space is always locally compact and $\sigma$-compact. \emph{We also fix a Radon measure $\mu$ on $(X,d)$ with full support} (i.e., $\mu$ is a regular Borel measure on $X$ taking finite values on compact subsets, and for each $x \in X$, there exists a neighbourhood $U$ of $x$ such that $\mu(U)>0$).

\section{Preliminaries}\label{preliminaries}
In this section, we provide the background settings of this paper by collecting several basic notions from Banach algebra theory and coarse geometry. Throughout the section, let $E$ be a (complex) Banach space and $(X,d,\mu)$ be a proper metric space with a Radon measure $\mu$ on $X$ of full support. 

Denote $C_b(X)$ the space of bounded continuous functions on $X$, $C_0(X)$ the space of continuous functions on $X$ vanishing at infinity, and $C_c(X)$ the space of continuous functions on $X$ with compact supports.

\subsection{Banach space valued $L^p$-spaces}
In this subsection, we recall some basic notions and facts on Banach space valued $L^p$-spaces.
\begin{defn}
Let $p \in [1,\infty]$. For a \emph{Bochner measurable} function (i.e., it equals $\mu$-almost everywhere to a pointwise limit of a sequence of simple functions)\footnote{It follows from Pettis measurability theorem that Bochner measurability agrees with weak measurability when the Banach space $E$ is separable.} $\xi: (X,\mu) \to E$, its \emph{$p$-norm} is defined by
$$\|\xi\|_p:=\big( \int_{X} \|\xi(x)\|_E^p \d\mu(x) \big)^{\frac{1}{p}},$$
and its \emph{infinity-norm} is defined by
$$\|\xi\|_\infty:=\mathrm{ess~sup}\{\|\xi(x)\|_E: x\in X\}.$$
For $p \in [1,\infty]$, the space of $E$-valued $L^p$-functions on $(X,\mu)$ is defined as follows:
$$L^p(X,\mu;E):=\big\{\xi:X \to E ~\big|~ \xi \mbox{~is~Bochner~measurable~and~} \|\xi\|_p < \infty\big\}\big/ \sim,$$
where $\xi \sim \eta$ if and only if they are equal $\mu$-almost everywhere. Equipped with the $p$-norm, $L^p(X,\mu;E)$ becomes a Banach space, which is called the \emph{$L^p$-Bochner space}.
\end{defn}
We also need the following closed linear subspace of $L^\infty(X,\mu;E)$:
\begin{equation*}
L^0(X,\mu;E):=\big\{[\xi] \in L^\infty(X,\mu;E)~\big|~ \forall \varepsilon>0, \exists \mbox{~compact~}K \subseteq X, \mbox{~s.t.~} \|\xi|_{X\setminus K}\|_\infty < \varepsilon\big\},
\end{equation*}
equipped with the norm $\|\xi\|_0:=\|\xi\|_\infty$. Clearly, $L^0(X,\mu;E)$ contains $C_0(X)$ but is more flexible, as it also contains all characteristic functions of bounded subsets of the proper metric space $(X,d)$. On the other hand, $L^0(X,\mu;E)$ inherits some nice behaviours of $C_0(X)$, for example a representative can always be chosen for each element in $L^0(X,\mu;E)$ such that its norm goes to zero when the variable goes to infinity.

In order to simplify notations, we regard $\xi$ as an element in $L^p(X,\mu;E)$ and write $\LpE$ instead if there is no ambiguity. If $X$ is discrete and equipped with the counting measure, we simply write $\ell^p(X;E)$.

If $p \in (1,\infty)$, let $q$ be the \emph{conjugate exponent} to $p$ (i.e., $\frac{1}{p}+\frac{1}{q}=1$) and if $p=1$, we set $q=0$ instead of $q=\infty$. It is worth noticing that the duality $L^p(X;E)^* \cong L^q(X;E^*)$ does not hold in general (see e.g. \cite{MR1027088, MR0453964, MR0206190}), but we still have the following lemma.

\begin{lem}\label{norm isom}
When $p \in (1,\infty)$, set $q$ to be its conjugate exponent and when $p=1$, set $q=0$. Then there is an isometric embedding $\LqEst \to \LpE^*$ defined by
$$\eta(\xi):=\int_X \eta(x)(\xi(x)) \d\mu(x)$$
where $\eta \in \LqEst$ and $\xi \in \LpE$. On the other hand, there is another isometric embedding $\LpE \to \LqEst^*$ defined by
$$\xi(\zeta):=\int_X \zeta(x)(\xi(x)) \d\mu(x)$$
where $\xi \in \LpE$ and $\zeta \in \LqEst$.
\end{lem}

\begin{proof}
For the first statement, when $p>1$ it follows from the same argument showing the classical result that $L^q(X;\mathbb{C})$ embeds isometrically into $L^p(X;\mathbb{C})^*$ (which are indeed isomorphic). And for $p=1$, we have the following maps
$$L^0(X;E^*) \subseteq L^\infty(X;E^*) \hookrightarrow L^1(X,E)^*$$
where the second isometric embedding follows from the same argument showing the classical result that $L^\infty(X;\mathbb{C})$ embeds isometrically into $L^1(X;\mathbb{C})^*$.

For the second statement, it suffices to show that for any $\xi \in \LpE$, we have
$$\|\xi\|_p = \sup\{|\xi(\zeta)|: \zeta \in \LqEst \mbox{~and~} \|\zeta\|_q \leq 1\}.$$
It is clear that the right hand side does not exceed the left. Conversely we may assume, by the inner regularity of $\mu$, that $\xi$ is non-zero and $\xi=\sum_{i=1}^{n} y_i \chi_{\Omega_i}$ for some $y_i\in E$ and mutually disjoint compact subsets $\Omega_i$ in $X$. Note that $\| \xi\|_p^p=\sum_{i=1}^n \|y_i\|_E^p\mu(\Omega_i)$. For each $y_i$, choose a $y_i^* \in (E^*)_1$ such that $y_i^*(y_i)=\|y_i\|_E$. Define
$$\zeta:=\sum_{i=1}^{n} \frac{\|y_i\|_E^{p-1}}{\|\xi\|_p^{p-1}}y_i^* \chi_{\Omega_i}.$$
Note that when $p=1$, $\zeta$ can be written simply as $\sum_{i=1}^{n} y_i^* \chi_{\Omega_i}$. It is straightforward to check that $\zeta \in \LqEst$ with $\|\zeta\|_q=1$ and $\xi(\zeta)=\|\xi\|_p$ (note that when $p=1$, we set $q=0$). Hence, we finish the proof.
\end{proof}

Finally we recall $L^p$-tensor products (more details can be found in \cite[Chapter 7]{DF}, \cite[Theorem 2.16]{Phil12} and \cite{michor2006functors}), which will be used in Section~\ref{$L^p$-Roe-like algebras} without further reference.

For $p\in[1,\infty)$, there is a tensor product of $L^p$-spaces with $\sigma$-finite measures such that there is a canonical isometric isomorphism $L^p(X,\mu)\otimes L^p(Y,\nu)\cong L^p(X\times Y,\mu\times\nu)$, which identifies the element $\xi\otimes\eta$ with the function $(x,y)\mapsto\xi(x)\eta(y)$ on $X\times Y$ for every $\xi\in L^p(X,\mu)$ and $\eta\in L^p(Y,\nu)$. Moreover, the following properties hold:
\begin{itemize}
\item Under the identification above, the linear spans of all $\xi\otimes\eta$ are dense in $L^p(X\times Y,\mu\times\nu)$.
\item $||\xi\otimes\eta||_p=||\xi||_p||\eta||_p$ for all $\xi\in L^p(X,\mu)$ and $\eta\in L^p(Y,\nu)$.
\item The tensor product is commutative and associative.
\item If $a\in \mathfrak{B}(L^p(X_1,\mu_1),L^p(X_2,\mu_2))$ and $b\in \mathfrak{B}(L^p(Y_1,\nu_1),L^p(Y_2,\nu_2))$, then there exists a unique element \[c\in \mathfrak{B}(L^p(X_1\times Y_1,\mu_1\times\nu_1),L^p(X_2\times Y_2,\mu_2\times\nu_2))\] such that under the identification above, $c(\xi\otimes\eta)=a(\xi)\otimes b(\eta)$ for all $\xi\in L^p(X_1,\mu_1)$ and $\eta\in L^p(Y_1,\nu_1)$. We will denote this operator by $a \otimes b$. Moreover, $||a\otimes b||=||a|| \cdot ||b||$.
\item The above tensor product of operators is associative, bilinear, and satisfies $(a_1\otimes b_1)(a_2\otimes b_2)=a_1a_2\otimes b_1b_2$.
\end{itemize}
If $A\subseteq \mathfrak{B}(L^p(X,\mu))$ and $B\subseteq \mathfrak{B}(L^p(Y,\nu))$ are closed subalgebras, we define $A\otimes B\subseteq \mathfrak{B}(L^p(X\times Y,\mu\times\nu))$ to be the closed linear span of all $a\otimes b$ with $a\in A$ and $b\in B$.

\subsection{Block cutdown maps}
Now we introduce block cutdown maps, providing an approach to cut an operator into the form of block diagonals.

First, let us recall some more notions. For $p \in \{0\} \cup [1,\infty]$, the \emph{multiplication representation} $\rho:C_b(X) \rightarrow \BLpE$ is defined by pointwise multiplications: $(\rho (f)\xi)(x)=f(x)\xi(x)$, where $f \in C_b(X)$, $\xi \in \LpE$ and $x \in X$. Without ambiguity, we write $fT$ and $Tf$ instead of $\rho(f)T$ and $T\rho(f)$, for $f\in C_b(X)$ and $T \in \BLpE$, respectively. It is worth noticing that $\mu$ has full support if and only if $\rho$ is injective. We also recall that a net $\{T_\alpha\}$ converges in \emph{strong operator topology (SOT)} to $T$ in $\BLpE$ if and only if $\|T_\alpha(\xi)-T(\xi)\|_p  \rightarrow 0$ for any $\xi \in \LpE$.

\begin{defn}\label{block cut down}
Given an equicontinuous family $(e_j)_{j \in J}$ of positive contractions in $C_b(X)$ with pairwise disjoint supports, define the \emph{block cutdown} map $\theta_{(e_j)_{j \in J}}: \BLpE \rightarrow \BLpE$ by
\begin{equation}\label{defn for theta}
\theta_{(e_j)_{j \in J}}(a):=\sum_{j \in J}e_jae_j,
\end{equation}
where the sum converges in (SOT) by Lemma~\ref{SOT convergence lem} below. We say that a closed subalgebra $B \subseteq \BLpE$ is \emph{closed under block cutdowns}, if $\theta_{(e_j)_{j \in J}}(B)\subseteq B$ for every equicontinuous family $(e_j)_{j \in J}$ of positive contractions in $C_b(X)$ with pairwise disjoint supports.
\end{defn}

\begin{lem}\label{SOT convergence lem}
Let $(e_j)_{j \in J}$ and $(f_j)_{j \in J}$ be two equicontinuous families of positive contractions in $C_b(X)$ with pairwise disjoint supports, and $a \in \BLpE$. Then the sum $\sum_{j \in J}f_jae_j$ converges in (SOT) to an operator in $\BLpE$. Furthermore, we have:
$$\big\|\sum_{j \in J}f_jae_j\big\|=\sup_{j \in J}\|f_jae_j\|.$$
\end{lem}

\begin{proof}
First of all, we prove in the case of $p\in (1,\infty)$ and let $q$ be the conjugate exponent to $p$. Let $Y_j:=\supp(e_j)$ and $Z_j:=\supp(f_j)$. For any $\xi \in \LpE$, any finite subset $F \subseteq J$ and any $\eta \in \LqEst$ with $\|\eta\|_q \leq 1$, we have that
\begin{eqnarray*}
  \big|\eta\big(\sum_{j \in F}f_jae_j\xi\big)\big| &=& \big|\sum_{j \in F} (\eta\chi_{Z_j})(f_jae_j\chi_{Y_j}\xi)\big| \\
   &\leq &\sum_{j \in F} \|\eta|_{Z_j}\|_q\cdot\|f_jae_j(\xi|_{Y_j})\|_p\\
   &\leq &\big(\sum_{j \in F} \|\eta|_{Z_j}\|_q^q\big)^{\frac{1}{q}}\cdot \big(\sum_{j \in F} \|f_jae_j(\xi|_{Y_j})\|_p^p\big)^{\frac{1}{p}}\\
   & \leq & \sup_{j \in J}\|f_jae_j\| \cdot \|\xi|_{\sqcup_{j\in F}Y_j}\|_p.
\end{eqnarray*}
Hence, it follows from Lemma \ref{norm isom} that
$$\big\| \big(\sum_{j \in F}f_jae_j\big)\xi \big\|_p \leq \sup_{j \in J}\|f_jae_j\| \cdot \|\xi|_{\sqcup_{j\in F}Y_j}\|_p.$$
Since $\|\xi|_{\sqcup_{j\in J}Y_j} \|_p\leq \|\xi\|_p<\infty$, we know $\big\{\xi|_{\sqcup_{j\in F}Y_j}\big\}_F$ is a Cauchy net. Hence, $\sum_{j \in J}f_jae_j$ converges in (SOT) and $\|\sum_{j \in J}f_jae_j\| \leq \sup_{j \in J}\|f_jae_j\|$. On the other hand, it is clear that $\|\sum_{j \in J}f_jae_j\|$ $\geq \sup_{j \in J}\|f_jae_j\|$. Hence we finish the proof for $p>1$. Since the proof for the case of $p=1$ is more direct, we leave the details to the reader.
\end{proof}

\begin{rem}
Note that the multiplication by $C_b(X)$ commutes with the block cutdowns, i.e., for any $a \in \BLpE$ and $f \in C_b(X)$, we have
$$f\theta_{(e_j)_{j \in J}}(a)=\theta_{(e_j)_{j \in J}}(fa) \andx \theta_{(e_j)_{j \in J}}(a)f=\theta_{(e_j)_{j \in J}}(af).$$
\end{rem}

%\begin{rem}
%Note that for the above family $(e_j)_{j \in J}$, the bounded function on $X$ defined by
%$$e:=\sum_{j \in J}e_j$$
%might not be continuous. However, if the family $(e_j)_{j \in J}$ is equicontinuous, then $e$ is continuous. On the other hand, if there exists some constant $R>0$ such that supports of $(e_j)_{j \in J}$ are mutually $R$-disjoint, then $e$ is also continuous.
%\end{rem}

\begin{defn}\label{block diag}
Suppose $\mathcal{X}$ is a metric family of subsets in $X$ (recall that a \emph{metric family} is a set of metric spaces), each of the subset is equipped with the induced metric
%(i.e., at most countable sets of subsets of $X$)
and $a \in \BLpE$. We say that $a$ is \emph{block diagonal with respect to $\mathcal{X}$}, if there exist an equicontinuous family $(e_j)_{j\in J}$ of positive contractions in $C_b(X)$ with pairwise disjoint supports and $\{Y_j\}_{j\in J} \subseteq \mathcal{X}$, such that
$$a=\theta_{(e_j)_{j \in J}}(a),$$
and $\supp(e_j) \subseteq Y_j$. In this case, we shall denote $a_{Y_j}:=e_jae_j$, which is called the \emph{$Y_j$-block of $a$}.
\end{defn}

\subsection{$L^p$-Roe-like algebras}\label{$L^p$-Roe-like algebras}
Now we introduce $L^p$-Roe-like algebras, which are our main objects in this paper.
\begin{defn}
Let $R \geq 0$ and $a \in \BLpE$. We say that
\begin{itemize}
  \item $a$ has \emph{propagation at most $R$}, if for any $f,f' \in C_b(X)$ with $d(\supp(f),\supp(f')) > R$, then $faf'=0$.
  \item $a$ has \emph{$\varepsilon$-propagation at most $R$} for some $\varepsilon >0$,  if for any $f,f' \in C_b(X)_1$ with $d(\supp(f),\supp(f')) > R$, then $\|faf'\|< \varepsilon$.
  \item $a$ is \emph{quasi-local}, if it has finite $\varepsilon$-propagation for every $\varepsilon>0$.
\end{itemize}
\end{defn}

\begin{defn}\label{$L^p$-Roe-like algebras def}
Let $(X,d)$ be a proper metric space equipped with a Radon measure $\mu$ whose support is $X$, and $p \in [1,+\infty)$. Suppose $E$ is a Banach space and $B \subseteq \BLpE$ is a Banach subalgebra such that $C_b(X)BC_b(X)=B$ and is closed under block cutdowns. Define:
\begin{enumerate}[(i)]
  \item $\Roe$ to be the norm-closure of all the operators in $B$ with finite propagations. $\Roe$ is called the \emph{$L^p$-Roe-like algebra of $(X,d,\mu)$};
  \item $\K$ to be the norm-closure of $C_0(X)BC_0(X)$ in $\BLpE$.
\end{enumerate}
\end{defn}

\begin{rem}\label{weak def rem}
The definition of $L^2$-Roe-like algebras come from \cite[Definition~2.3]{spakula2017relative}, in which the following extra condition is also imposed:
\begin{equation}\label{commut condition 2.2}
  [C_0(X),B] \subseteq \K.
\end{equation}
This condition is used in the proof of their main theorem, \cite[Theorem 2.8, ``(i) $\Rightarrow$ (iii)"]{spakula2017relative}. However, it turns out to be redundant if we apply our Lemma~\ref{2.2 replacement} below. On the other hand, this condition is fulfilled by most of the well-known $L^p$-Roe-like algebras for $p\in(1,\infty)$ (as we will see in the following examples), but not for $p=1$ (see the explanation in Example~\ref{uniform Roe alg}). This is exactly our starting point to explore whether condition \eqref{commut condition 2.2} is necessary, and it turns out that we may omit it in Definition \ref{$L^p$-Roe-like algebras def} without affecting the main theorem. In this way, our main result (Theorem~\ref{char for Roe alg Thm}) is a slight generalisation of \cite[Theorem~2.8]{spakula2017relative}.
\end{rem}

We notice that in the case of $p=2$, it has been pointed out in \cite[Remark~2.4]{spakula2017relative} that $\K$ is an ideal in $\Roe$ under the additional condition (\ref{commut condition 2.2}). Now we show that it still holds in our settings.

\begin{lem}\label{ideal}
For any $p \in [1,+\infty)$, $\K$ is a closed two-sided ideal in $\Roe$.% provided that $C_b(X)BC_b(X)=B$.
\end{lem}

\begin{proof}
It suffices to show that for any $b=f_1b_1g_1 \in C_c(X)BC_c(X)$ and $a\in B$ with finite propagation at most $R$, $ba \in \K$. Take a function $g_2 \in C_c(X)$ such that $g_2$ is $1$ on the compact subset $\overline{\N_R(\supp(g_1))}$. It follows that $g_1a(1-g_2)=0$, which implies that $g_1a=g_1ag_2$. Hence, we have
$$ba=f_1b_1g_1a=f_1(b_1g_1a)g_2.$$
Recall that $C_b(X)BC_b(X)=B$, so we have $b_1g_1\in B$ and $a \in B$, which implies that $ba \in C_c(X)BC_c(X)$. Similarly, $ab \in C_c(X)BC_c(X)$ as well. So we finish the proof.
\end{proof}

Before we illustrate several examples of $L^p$-Roe-like algebras, let us recall the following notion related to matrix algebras.

\begin{defn}
Let $(X,d)$ be a discrete proper metric space and $p \in [1,+\infty)$. Denote
$$\overline{M}^p_X:=\overline{C_c(X)\mathfrak{B}(\ell^p(X))C_c(X)}^{\mathfrak{B}(\ell^p(X))},$$
i.e., for any fixed point $x_0 \in X$
$$\overline{M}^p_X=\overline{\bigcup_{n\in \mathbb{N}}M^p_{B_n(x_0)}}$$
where $M^p_{B_n(x_0)}=\mathfrak{B}(\ell^p(B_n(x_0))) \subseteq \mathfrak{B}(\ell^p(X))$, which is the matrix algebra over the closed ball of radius $n$ and centered in $x_0$. In other words, operators in $\overline{M}^p_X$ are exactly those that can be approximated by finite matrices.
\end{defn}

Phillips studied the relation between $\overline{M}^p_X$ and compact operators $\mathfrak{K}(\ell^p(X))$ in \cite{phillips2013crossed}. He showed that when $p>1$, $\overline{M}^p_X=\mathfrak{K}(\ell^p(X))$ (\cite[Lemma 1.7, Corollary 1.9]{phillips2013crossed}); and when $p=1$, $\overline{M}^1_X \subsetneq \mathfrak{K}(\ell^1(X))$ in general as illustrated in \cite[Example 1.10]{phillips2013crossed} (see also Example~\ref{uniform Roe alg}).

Now we are ready to provide various of examples of $L^p$-Roe-like algebras, which include $\ell^p$-uniform Roe algebras, band-dominated operator algebras, $L^p$-Roe algebras, $\ell^p$-uniform algebras and stable $\ell^p$-uniform Roe algebras.

\begin{ex}[\textbf{$\ell^p$-Uniform Roe Algebra}]\label{uniform Roe alg}
Let $(X,d)$ be a discrete proper metric space and $p \in [1,+\infty)$. Take $E=\mathbb{C}$ to be the complex number and $B=\mathfrak{B}(\ell^p(X))$, which is clearly closed under block cutdowns, and satisfies $C_b(X)BC_b(X)=B$. In this case, $\Roe$ is called \emph{the $\ell^p$-uniform Roe algebra} of $X$, which is defined in \cite{chung2018rigidity} and denoted by $B^p_u(X)$, and $\K$ is $\overline{M}^p_X$ introduced above. It may be worth noting that $\overline{M}^1_X$ is structurally different from $\overline{M}^p_X$ for $p>1$.

$\bullet~ p>1$:
As pointed out above, $\overline{M}^p_X=\mathfrak{K}(\ell^p(X))$. And condition (\ref{commut condition 2.2}) follows from the fact that $C_0(X)B \subseteq \K$ and $BC_0(X) \subseteq \K$.

$\bullet~ p=1$:
The algebra $\K$ is in general \emph{properly} contained in $\mathfrak K (\ell^1(X))$ (see Example~1.10 in \cite{phillips2013crossed}). For example, taking $X$ to be the natural number $\mathbb{N}$, consider the operator $T: \ell^1(\mathbb{N}) \rightarrow \ell^1(\mathbb{N})$ defined by
$$T(\xi):=\big(\sum_{n\in \mathbb{N}} \xi(n)\big) \delta_0,$$
where $\xi \in \ell^1(\mathbb{N})$ and $\delta_0 \in \ell^1(\mathbb{N})$ is the function taking value $1$ at the origin point $0$, and $0$ elsewhere. Since $T$ has rank $1$, it belongs to $\mathfrak K (\ell^1(\mathbb{N}))$. However, it is not hard to see that $T\notin \mathcal{K}(\mathbb{N}, \mathfrak{B}(\ell^1(\mathbb{N})))=\overline{M}^1_{\mathbb{N}}$. Furthermore, the operator $T$ also illuminates that condition (\ref{commut condition 2.2}) does \emph{not} hold in general, since $[\delta_0,T] \notin \mathcal{K}(\mathbb{N}, \mathfrak{B}(\ell^1(\mathbb{N})))$.
\end{ex}

\begin{ex}[\textbf{Band-Dominated Operator Algebra}]
Let $(X,d)$ be a uniformly discrete metric space of \emph{bounded geometry} (in the sense that for a given $R>0$, all closed balls $B(x,R)$ have a uniform bound on cardinalities for all $x\in X$), $p \in (1,+\infty)$ and $E$ be a Banach space. Take $B=\mathfrak{B}(\ell^p(X;E))$, which is clearly closed under block cutdowns and satisfies $C_b(X)BC_b(X)=B$. Elements in $B$ can be represented in the matrix form
$$b=(b_{x,y})_{x,y\in X} \in \mathfrak{B}(\ell^p(X;E)), \quad \mbox{where} \quad b_{x,y} \in \mathfrak{B}(E).$$
In this case, $\Roe=\mathcal{A}^p_E(X)$, which is the algebra of band-dominated operators (see \cite[Definition~2.6]{vspakula2017metric}) and it is clear that $\K=\mathcal{K}^p_E(X)$, which is the set of all $\mathcal{P}$-compact operators on $\ell^p(X;E)$, defined in \cite[Definition~2.8]{vspakula2017metric}.
\end{ex}

\begin{ex}[\textbf{$L^p$-Roe Algebra}]\label{$L^p$-Roe Algebra}
Let $(X,d)$ be a proper metric space equipped with a Radon measure $\mu$ with support $X$, and $p \in [1,+\infty)$. We say that an operator $b$ in $\mathfrak{B}(L^p(X; \ell^p(\mathbb{N}))) \cong \mathfrak{B}(L^p(X \times \mathbb{N}))$ is \emph{locally compact} if for any $f \in C_0(X)$, $fb$ and $bf$ belong to $\mathfrak{K}(L^p(X \times \mathbb{N}))$.

Now take $E=\ell^p(\mathbb{N})$ and $B$ to be the set of all locally compact operators in $\mathfrak{B}(L^p(X; \ell^p(\mathbb{N})))$, which is clearly closed under block cutdowns and satisfies $C_b(X)BC_b(X)=B$. The corresponding $L^p$-Roe-like algebra $\Roe$ is called \emph{the $L^p$-Roe algebra of $X$}, denoted by $B^p(X)$. It is, by definition, the norm closure of all locally compact and finite propagation operators in $\mathfrak{B}(L^p(X; \ell^p(\mathbb{N})))$. Analogous to the arguments in Example \ref{uniform Roe alg}, one can check that when $p>1$, $\K=\mathfrak{K}(L^p(X \times \mathbb{N}))$ and it does not hold in general when $p=1$. 

When $X$ is discrete, the $L^p$-Roe algebra $B^p(X)$ coincides with the $\ell^p$-Roe algebra defined in \cite{chung2018rigidity} and in the case of $p=2$, the $L^2$-Roe algebra is the classical Roe algebra in the literature.
\end{ex}

\begin{rem}\label{weak version of L1 roe alg}
As explained in \cite{chung2018rigidity}, there is another version of locally compactness: we say that an operator $b$ in $\mathfrak{B}(L^p(X; \ell^p(\mathbb{N})))$ is \emph{locally compact} if for any $f \in C_0(X)$, $fb$ and $bf$ belong to $\mathfrak{K}(L^p(X)) \otimes \overline{M}^p_{\mathbb{N}} \subseteq \mathfrak{B}(L^p(X \times \mathbb{N}))$. Note that the subalgebra $\mathfrak{K}(L^p(X)) \otimes \overline{M}^p_{\mathbb{N}}$ is isomorphic to the norm closure of $\bigcup_{n\in \mathbb{N}} M^p_n(\mathfrak{K}(L^p(X)))$. Therefore, we can alternatively define another version of the $L^p$-Roe algebra of $X$ to be the norm closure of all locally compact (in this new sense) and finite propagation operators in $\mathfrak{B}(L^p(X; \ell^p(\mathbb{N})))$. When $p>1$, it coincides with $B^p(X)$ defined in Example~\ref{$L^p$-Roe Algebra} as $\mathfrak{K}(L^p(X)) \otimes \overline{M}^p_{\mathbb{N}} \cong \mathfrak{K}(L^p(X \times \mathbb{N}))$. However, it is strictly contained in $B^1(X)$ when $p=1$.
\end{rem}

\begin{ex}[\textbf{$\ell^p$-Uniform Algebra}]
Let $(X,d)$ be a discrete metric space with bounded geometry and $p \in [1,+\infty)$. Set $E=\ell^p(\mathbb{N})$, and $B$ to be the closure of the set of all $b=(b_{x,y})_{x,y \in X} \in \mathfrak{B}(\ell^p(X; \ell^p(\mathbb{N})))$ for which the rank of $b_{x,y} \in \mathfrak{B}(\ell^p(\mathbb{N}))$ is uniformly bounded. Clearly, $B$ is closed under block cutdowns, and satisfies $C_b(X)BC_b(X)=B$. In this case, $\mathrm{Roe}(X,B)=UB^p(X)$, \emph{the $\ell^p$-uniform algebra of $X$}, introduced in \cite{chung2018rigidity}. When $p>1$, we have that $\mathcal{K}(X,B)=\mathfrak{K}(l^p(X \times \mathbb{N}))$. But it does not hold in general when $p=1$.
\end{ex}

\begin{ex}[\textbf{Stable $\ell^p$-Uniform Roe Algebra}]\label{Stable lp Uniform Roe Algebra}
Let $(X,d)$ be a discrete metric space with bounded geometry and $p \in [1,+\infty)$. Set $E=\ell^p(\mathbb{N})$, and $B$ to be the closure of the set of all $b=(b_{x,y})_{x,y \in X} \in \mathfrak{B}(\ell^p(X; \ell^p(\mathbb{N})))$ for which there exists a finite-dimensional subspace $E_b \subseteq \ell^p(\mathbb{N})$ such that $b_{x,y} \in \mathfrak{B}(E_b) \subseteq \mathfrak{B}(\ell^p(\mathbb{N}))$. Clearly, $B$ is closed under block cutdowns and satisfies $C_b(X)BC_b(X)=B$. In this case, $\mathrm{Roe}(X,B)=B^p_s(X)$, \emph{the stable $\ell^p$-uniform Roe algebra of $X$}, introduced in \cite{chung2018rigidity}. Moreover, $B^p_s(X)\cong B^p_u(X) \otimes \mathfrak{K}(\ell^p(\mathbb{N}))$, which explains the terminology. Analogous to the arguments in Example \ref{uniform Roe alg}, one can check that when $p>1$, $\mathcal{K}(X,B)=\mathfrak{K}(l^p(X \times \mathbb{N}))$ and it does not hold in general when $p=1$.
\end{ex}

\begin{rem}
As explained in \cite{chung2018rigidity}, there is another version of the stable $\ell^p$-uniform Roe algebra of $X$, defined to be the norm closure of finite propagation operators $b=(b_{x,y})_{x,y \in X} \in \mathfrak{B}(\ell^p(X; \ell^p(\mathbb{N})))$ for which there exists some $k\in \mathbb{N}$ such that $b_{x,y} \in M_k(\mathbb{C}) \subseteq \mathfrak{B}(\ell^p(\mathbb{N}))$ (here $M_k(\mathbb{C})$ is embedded as a subalgebra of $\mathfrak{B}(\ell^p(\mathbb{N}))$ in a fixed way, independent of the points in $X$). It is clear that this algebra is isomorphic to $B^p_u(X) \otimes \overline{M}^p_\mathbb{N}$ for all $p\in [1,\infty)$. As before for $p>1$, it coincides with $B^p_s(X)$ defined in Example~\ref{Stable lp Uniform Roe Algebra}.
\end{rem}

\begin{rem}
In general, we have that for discrete space $X$,
$$B^p_u(X) \subseteq B^p_s(X) \subseteq UB^p(X) \subseteq B^p(X).$$
It is worth noticing that $UB^1(X)$ is not contained in the weak version of the $L^1$-Roe algebra defined in Remark \ref{weak version of L1 roe alg}. Indeed, Example~\ref{uniform Roe alg} provides a rank one operator $T \in \mathfrak{B}(\ell^1(\mathbb{N}))$ which does not sit in $\overline{M}^1_\mathbb{N}$. Define the diagonal operator $b \in \mathfrak{B}(\ell^1(X; \ell^1(\mathbb{N})))$ by $b_{x,x}:=T$ for $x\in X$, and $b_{x,y}=0$ for $x \neq y$. Clearly, $b$ is such an example as desired.
\end{rem}

\subsection{Straight finite decomposition complexity} In this subsection, we explain the notion of straight finite decomposition complexity, which will be used in the sequel.

Straight finite decomposition complexity (sFDC) was introduced in \cite{dranishnikov2014asymptotic} as a weak version of the original notion of finite decomposition complexity (FDC), which was introduced and studied by Guentner, Tessera and Yu in their study of topological rigidity in \cite{guentner2012notion}. In general, finite asymptotic dimension implies finite decomposition complexity \cite[Theorem~4.1]{MR3054574}, which consequently implies straight finite decomposition complexity \cite[Proposition~2.3]{dranishnikov2014asymptotic}. Moreover, it was also shown in \cite[Theorem~3.4]{dranishnikov2014asymptotic} that straight finite decomposition complexity does imply Yu's Property A. However, it is still unknown whether (FDC), (sFDC) and Yu's Property A are all equivalent or not.

\begin{defn}
Let $(X,d)$ be a proper metric space and $Z, Z' \subseteq X$. Let $\Xs,\Ys$ be metric families of subsets in $X$, and $R\geq 0$.
\begin{itemize}
  \item $\mathcal{X}$ is \emph{uniformly bounded}, if $\sup_{X\in \Xs}\mathrm{diam}(X)<\infty.$
  \item Denote the \emph{$R$-neighbourhood of $Z$} by $\N_R(Z):=\{z\in X: d(z,Z)\leq R\}$. Set
      $$\N_R(\Xs):=\{\N_R(X):X\in\Xs \}.$$
  \item A metric family $(Y_j)_{j\in J}$ of subsets of $X$ is \emph{$R$-disjoint}, if $d(Y_j,Y_j')>R$ for all $j\neq j'$. Write $$\bigsqcup_{R-disjoint}Y_j$$
      for their union to indicate that the family is $R$-disjoint.
  \item $Z$ can \emph{$R$-decompose over $\Ys$}, if $Z$ can be decomposed into $Z=X_0\cup X_1$ and
  $$X_i=\bigsqcup_{R-\text{disjoint}}X_{ij}, \quad i=0,1,$$
  such that $X_{ij}\in\Ys$ for all $i,j.$
  \item $\Xs$ can \emph{$R$-decompose over $\Ys$}, denoted by $\Xs\xrightarrow{R}\Ys$, if every $Y\in\Xs$ can $R$-decompose over $\Ys$.
  \item $X$ has \emph{straight finite decomposition complexity}, if for any sequence $0\leq R_1<R_2< \cdots,$ there exists $m\in\mathbb{N}$ and metric families $\{X\}=\Xs_0,\Xs_1,\ldots,\Xs_m$, such that $\Xs_{i-1} \xrightarrow{R_i} \Xs_i$ for $i=1,\ldots, m$, and the family $\Xs_m$ is uniformly bounded.
\end{itemize}
\end{defn}

We remark here briefly that one way to define finite decomposition complexity is to use a ``decomposition game", which means \emph{a priori} that the choices of $R_i$ might depend on the previous families $\Xs_0,\Xs_1,\ldots,\Xs_{i-1}$. Consequently, sFDC can be obviously implied from FDC.

\section{The main theorem}
In this section, we present our main result (Theorem~\ref{char for Roe alg Thm}), which gives several different pictures of how elements in $L^p$-Roe-like algebras may look like. We also prove the relatively easier part where straight finite decomposition complexity is not required, while leaving the rest of the proof to the next section after more technical tools are developed.

To state our main theorem, we need to introduce some notions as follows.

\begin{defn}[\cite{roe2003lectures}]
Let $(X,d)$ be a proper metric space. A function $g \in C_b(X)$ is called a \emph{Higson function} (also called a \emph{slowly oscillating function}), if for every $R>0$ and $\varepsilon>0$, there exists a compact set $A \subseteq X$ such that for any $x,y \in X\backslash A$ with $d(x,y) < R$, then $|g(x)-g(y)| < \varepsilon$. The set of all Higson functions on $X$ is denoted by $C_h(X)$.
\end{defn}

\begin{defn}\cite[Definition 2.6]{spakula2017relative},
Let $(X,d)$ be a metric space. A bounded sequence $(f_n)_{n \in \mathbb N}$ in $C_b(X)$ is called \emph{very Lipschitz}, if for every $L>0$, there exists $n_0 \in \mathbb N$ such that $f_n$ is $L$-Lipschitz for any $n \geq n_0$. Let $\VL$ denote the set of all very Lipschitz bounded sequences in $C_b(X)$. Define
$$\VLinf:=\VL \big/ \big\{(f_n)_{n\in \mathbb N} \in \VL: \lim_{n \rightarrow \infty} \|f_n\|=0\big\}.$$
\end{defn}

It is known from \cite{spakula2017relative} that $\VL$ is a $C^*$-subalgebra of $\ell^\infty(\mathbb{N},C_b(X))$ and $\VLinf$ is a $C^*$-subalgebra of $(C_b(X))_\infty$. In the following, we will view both $\VLinf$ and $B \subseteq \BLpE$ as Banach subalgebras of $\BLpE_\infty$, and consider the \emph{relative commutant}:
$$B \cap \VLinf'=\{b\in B: b \mbox{~commutes~with~elements~in~}\VLinf \}.$$
It is clear that any operator in $\BLpE$ with finite propagation commutes with $\VLinf$. Hence, by taking limits it follows that
\begin{equation}\label{iv to i}
\Roe \subseteq B \cap \VLinf'.
\end{equation}
The converse inclusion is also true provided the space $X$ has straight finite decomposition complexity and this is included in our main theorem as follows (this is the complete version of Theorem A in Section 1):

\begin{thm}\label{char for Roe alg Thm}
Let $(X,d)$ be a proper metric space equipped with a Radon measure $\mu$ whose support is $X$, and $p \in [1,+\infty)$. Suppose $E$ is a Banach space and $B \subseteq \BLpE$ is a Banach subalgebra such that $C_b(X)BC_b(X)=B$ and $B$ is closed under block cutdowns. Then for $b \in B$, the following are equivalent:
\begin{enumerate}
  \item[\emph{(i).}] $[b,f]=0$ for all $f \in \VLinf$;
  \item[\emph{(ii).}] $b$ is quasi-local;
  \item[\emph{(iii).}] $[b,g] \in \K$ for any $g \in C_h(X)$.
\end{enumerate}
If $X$ has straight finite decomposition complexity, then these are also equivalent to:
\begin{enumerate}
  \item[\emph{(iv).}] $b \in \Roe$.
\end{enumerate}
\end{thm}

Recall that we have already explained in Remark~\ref{weak def rem} that Theorem~\ref{char for Roe alg Thm} is a slight generalisation of \cite[Theorem~2.8]{spakula2017relative} as condition~(\ref{commut condition 2.2}) is not required here. Also notice that (\ref{iv to i}) implies that ``(iv) $\Rightarrow$ (i)" holds generally and the converse implication is also true under the extra condition of straight finite decomposition complexity.

In the remaining of this section, we prove that (i), (ii) and (iii) in Theorem~\ref{char for Roe alg Thm} are all equivalent, and leave the implication ``(i) $\Rightarrow$ (iv)" to the next section, after we develop some technical tools such as Proposition~\ref{estimate for cmmt cor}.

\subsection{``(i) $\Leftrightarrow$ (ii)"}

We start with the proof of Theorem~\ref{char for Roe alg Thm}, ``(i) $\Leftrightarrow$ (ii)". The implication ``(i) $\Rightarrow$ (ii)" follows exactly from the same arguments in \cite{spakula2017relative}, while the proof of ``(ii) $\Rightarrow$ (i)" is slightly different from that one given in \cite{spakula2017relative} due to the absence of inner products. Fortunately, since both proofs are relativity short, we include the details for the convenience of the reader.

Let us begin with the following characterisation of the condition (i) in Theorem~\ref{char for Roe alg Thm}, which is proved in \cite{spakula2017relative} when $p=2$ and actually holds for general $p$.
\begin{lem}\cite[Lemma~3.1]{spakula2017relative}\label{commut lemma}
Let $p\in [1,\infty)$, $b \in \BLpE$ and $\varepsilon>0$. Then $\|[b,f]\| < \varepsilon$ for every $f \in \VLinf_1$ \emph{if and only if} there exists some $L>0$ such that $b \in \Commut(L,\varepsilon)$, where
$$\Commut(L,\varepsilon):=\big\{a \in \BLpE: \|[a,f]\| < \varepsilon, \mbox{~for~any~}L\mbox{-Lipschitz~} f\in C_b(X)_1\big\}.$$
\end{lem}

\begin{proof}[Proof of Theorem \ref{char for Roe alg Thm}, ``(i)$\Leftrightarrow$ (ii)'']
Assume $b\in \BLpE$ such that $[b,\VLinf_1]=0$ and let $\varepsilon>0$. By Lemma \ref{commut lemma}, there exists some $L>0$ such that $b \in \Commut(L,\varepsilon)$. For any $f,g \in C_b(X)_1$ with $L^{-1}$-disjoint supports, we may choose an $L$-Lipschitz $h \in C_b(X)_1$ such that $h|_{\supp f} \equiv 1$ and $h|_{\supp g} \equiv 0$. In particular, $\|[b,h]\|< \varepsilon$. Therefore,
$$\|fbg\|=\|fhbg\|\leq \|[h,b]\|+\|fbhg\|<\varepsilon+0 = \varepsilon.$$
Hence, $b$ is quasi-local as desired.

On the other hand, we assume that for any $\varepsilon>0$, $b$ has finite $\varepsilon$-propagation. Without loss of generality, we may assume that $b$ is a contraction. Given $\varepsilon>0$, pick $N$ such that $6/N<\varepsilon/2$. By the hypothesis, $b$ has $\varepsilon/(2N^2)$-propagation at most $R>0$. For any $(2RN)^{-1}$-Lipschitz $f\in C_b(X)_1$, we claim that $\|[b,f]\|<\varepsilon$. In fact, take
$$A_1:=f^{-1}([0,\frac{1}{N}]), \andx A_i:=f^{-1}((\frac{i-1}{N},\frac{i}{N}]), \quad i=2,\ldots,N.$$
These sets partition $X$, and $A_i$ is $2R$-disjoint from $A_j$ for $|i-j|>1$. Now choose a partition of unity $e_1,\ldots,e_N \in C_b(X)$ such that $e_i$ is supported in $\N_{R/2}(A_i)$. Thus, $\|e_ibe_j\|< \varepsilon/(2N^2)$ for $|i-j|>1$. Meanwhile, we have
$$f \approx_{1/N} \sum_{i=1}^{N}\frac{i}{N}e_i.$$
Hence, it follows that
\begin{eqnarray*}
\|[f,b]\| &\leq&  \frac{2}{N} + \big\| \sum_{i=1}^N [\frac{i}{N}e_i,b]  \big\| \\
&=& \frac{2}{N} + \big\| \big(\sum_{i=1}^N \frac{i}{N}e_ib\big)\big( \sum_{j=1}^N e_j \big) - \big( \sum_{i=1}^N e_i \big) \big(\sum_{j=1}^N \frac{j}{N}be_j\big) \big\| \\
&\leq& \frac{2}{N} + \sum_{|i-j|>1} \|e_ibe_j\| + \left\| \sum_{|i-j|\leq 1}(\frac{i}{N}-\frac{j}{N})e_ibe_j \right\|.
\end{eqnarray*}
Each term in the first sum is dominated by $\frac{\varepsilon}{2N^2}$, hence $\sum_{|i-j|>1} \|e_ibe_j\| < \varepsilon/2$.
The second sum can be broken into four sums: note that the terms vanish when $i=j$; what remain are $j=i+1$ and $j=i-1$, and we break each of these further into even and odd parts. By Lemma \ref{SOT convergence lem}, each of these terms has norm at most $\frac{1}{N}$. Hence, we have that
$$\|[f,b]\|< \frac{2}{N}+\frac{\varepsilon}{2}+\frac{4}{N}< \varepsilon.$$
So we complete the proof by Lemma~\ref{commut lemma}.
\end{proof}

\subsection{``(i) $\Leftrightarrow$ (iii)"}
Now we move on to Theorem~\ref{char for Roe alg Thm}, ``(i) $\Leftrightarrow$ (iii)". Here our major work is focused on omitting condition \eqref{commut condition 2.2}, as well as providing a ``non-symmetric" version of the argument given in \cite{spakula2017relative} for $p=2$. However, the main body of the proof is still very similar to that of the original $p=2$ case \cite{spakula2017relative}, so we just outline the proof and highlight the differences we make here.

First of all, we recall that the proof of ``(i) $\Rightarrow$ (iii)" given in \cite{spakula2017relative} requires condition~(\ref{commut condition 2.2}):
$$
[C_0(X),B] \subseteq \K.
$$
After a careful reading of the proof, we realise that it is unnecessary to assume the entire $B$ essentially commuting with $C_0(X)$ but only a closed subalgebra of $B$ as shown in the following lemma:

\begin{lem}\label{2.2 replacement}
Let $p\in [1,\infty)$ and $B$ be a Banach subalgebra of $\BLpE$ such that $C_b(X)BC_b(X)=B$. If $b \in B$ satisfies $[b,\VLinf]=0$, then $[b,C_0(X)] \subseteq \K$.
\end{lem}

\begin{proof}
Let $b\in B$ such that $[b,\VLinf]=0$. Since $\K$ is closed, we only need to prove $[b,g] \in \K$ for any $g \in C_c(X)$.

Fix a base point $x_0 \in X$. For each $k \in \mathbb{N}$, we may choose a $(k^{-1})$-Lipschitz function $f_k\in C_b(X)_1$ such that $f_k$ vanishes on $\supp(g)$ and $f_k|_{B_{R_k}(x_0)^c}=1$ for some sufficiently large $R_k>0$. Hence, the sequence $(f_k)_{k \in \mathbb{N}}\in \VLinf$ and $\|[b,f_k]\| \to 0$ for $k \to \infty$ by assumption. Since $gf_k=0$ for any $k\in \mathbb{N}$, it follows that
\begin{align*}
  \|[b,g]f_k\| = \|bgf_k-gbf_k\|= \|gbf_k\|
   =\|g[b,f_k]\|
   \leq  \|g\|\cdot  \|[b,f_k]\| \to 0,
\end{align*}
as $k \to \infty$.
Similarly, we have that $\|f_k[b,g]\| \to 0$ and $\|f_k[b,g]f_k\| \to 0$ as $k \to \infty$.

Moreover, we have that
$$\|[b,g]-(1-f_k)[b,g](1-f_k)\| \leq \|[b,g]f_k\|+\|f_k[b,g]\|+\|f_k[b,g]f_k\| \to 0,$$
as $k \to \infty$. Since $\supp(1-f_k) \subseteq B_{R_k}(x_0)$ and $C_b(X)BC_b(X)=B$, it follows that $(1-f_k)[b,g](1-f_k) \in C_c(X)BC_c(X)$. Hence, $[b,g] \in \K$.
\end{proof}

Replacing condition~(\ref{commut condition 2.2}) by Lemma~\ref{2.2 replacement} in the original proof for $p=2$ \cite{spakula2017relative}, we obtain a proof of Theorem~\ref{char for Roe alg Thm} ``(i) $\Rightarrow$ (iii)" without any further changes. Hence we omit the details.

Now we outline the proof for the other direction, ``(iii) $\Rightarrow$ (i)". Since $L^p$-Roe-like algebras may not possess a bounded involution in general, the proof becomes slightly different. Fix a base point $x_0 \in X$. For $R>0$, we define $e_R \in C_0(X)$ by
$$e_R(x):=\max\{ 0,1-d(x,B_R(x_0))/R \}.$$

\begin{lem}[\cite{spakula2017relative}, Lemma 5.4]\label{VL to LH}
For $(f_k)_{k=1}^\infty \in \VL$, a subsequence $(f_{k_i})_{i=1}^\infty$, and a sequence of positive numbers $(R_i)_{i=0}^\infty$ with $R_{i+1} \geq 6R_i$ for each $i$, define
$$g_{(f_{k_i}),(R_i)} := \sum_{i=1}^{\infty} f_{k_i}(e_{R_i}-e_{3R_{i-1}}).$$
Then $g_{(f_{k_i}),(R_i)} \in C_h(X)$.
\end{lem}

\begin{proof}[Sketch of the proof of Theorem \ref{char for Roe alg Thm} ``(iii) $\Rightarrow$ (i)"]
Fix a point $x_0 \in X$, and set $B_R:=B_R(x_0)$. Recall $b$ satisfies the condition that $[b,g] \in \K$ for any $g \in C_h(X)$. Now assume $b$ is a contraction, and that there exists some $f=(f_k)_{k=1}^\infty \in \VLinf$ such that $[b,f] \neq 0$. Take an $\varepsilon: 0< \varepsilon< \|[b,f]\|$, and consider two cases.

\emph{Case I.} There exists $R_0>0$ such that for any $S>0$, there exists infinitely many $k$:
$$\mbox{either}\quad \|\chi_{B_{R_0}}[b,f_k](1-\chi_{B_S})\| > \frac{\varepsilon}{5}, \quad \mbox{or} \quad \|(1-\chi_{B_S})[b,f_k]\chi_{B_{R_0}}\| > \frac{\varepsilon}{5}.$$
In other words, there exists $R_0>0$ with the property that
\begin{enumerate}[1)]
  \item either: there exists a sequence $S_1<S_2<\ldots$ tending to $\infty$, such that for any $n \in \mathbb N$, there exist infinitely many $k$ such that $\|\chi_{B_{R_0}}[b,f_k](1-\chi_{B_{S_n}})\| > \frac{\varepsilon}{5}$;
  \item or: there exists a sequence $S_1<S_2<\ldots$ tending to $\infty$, such that for any $n \in \mathbb N$, there exist infinitely many $k$ such that $\|(1-\chi_{B_{S_n}})[b,f_k]\chi_{B_{R_0}}\| > \frac{\varepsilon}{5}$.
\end{enumerate}
We only prove in the first situation, while the second is similar.

Since $(f_k)$ is very Lipschitz, $f_k|_{B_{R_0}}$ tends towards constant as $k \to \infty$. So without loss of generality, we can assume that $f_k|_{B_{R_0}}\equiv\gamma_k$, for all $k$. Setting $\hat{f_k}:=f_k-\gamma_k$ gives us another very Lipschitz sequence $(\hat{f_k})_{k=1}^\infty$ satisfying the same condition and $\hat{f_k}|_{B_{R_0}}\equiv 0$. Additionally, for all $k$, we have that
$$\chi_{B_{R_0}}[b,\hat{f_k}]=\chi_{B_{R_0}}b\hat{f_k}.$$
By assumption, there exists a sequence $S_1<S_2<\ldots$ tending to $\infty$, such that for any $n \in \mathbb N$, we can find infinitely many $k$ such that $\|\chi_{B_{R_0}}b\hat{f_k}(1-\chi_{B_{S_n}})\| > \frac{\varepsilon}{5}$.

As the original proof for $p=2$ \cite[Section 5]{spakula2017relative}, we may choose sequences: $k_1 < k_2 < \ldots$ and $R_1,R_2,\ldots$ satisfying 
\begin{eqnarray*}
  \|\chi_{B_{R_0}}b\hat{f_{k_i}}(e_{R_i}-e_{3R_{i-1}})\| >  \frac{\varepsilon}{10}.
\end{eqnarray*}
Applying Lemma \ref{VL to LH} to $(\hat{f_{k_i}})$ and $(R_i)$, we obtain $g \in C_h(X)$ defined by:
$$g:=g_{(\hat{f_{k_i}}),(R_i)} = \sum_{i=1}^{\infty} \hat{f_{k_i}}(e_{R_i}-e_{3R_{i-1}}).$$
Hence for any $S>R_0$, choose an $i$ such that $3R_{i-1}>S$, and we have
\begin{eqnarray*}
  \|[b,g](1-\chi_{B_S})\| & \geq & \|\chi_{B_{R_0}}[b,g](1-\chi_{B_S})(\chi_{B_{2R_i}}-\chi_{B_{3R_{i-1}}})\| \\
  &= & \|\chi_{B_{R_0}}b\hat{f_{k_i}}(e_{R_i}-e_{3R_{i-1}})\| \\
  &> & \frac{\varepsilon}{10},
\end{eqnarray*}
which contradicts with the hypothesis that $[b,g] \in \K$.

\emph{Case II.} For every $R>0$, there exists $S>0$ such that, for all but finitely many $k \in \mathbb N$, we have
$$\|\chi_{B_R}[b,f_k](1-\chi_{B_S})\| \leq \frac{\varepsilon}{5} \andx \|(1-\chi_{B_S})[b,f_k]\chi_{B_R}\| \leq \frac{\varepsilon}{5}.$$
Without loss of generality, we may assume that $S>R$.

Suppose we are given $R>0$ and $K \in \mathbb N$, and let $S$ be given as above. Then there exists $k \geq K$ such that
\begin{equation}\label{commut est 1}
  \|\chi_{B_R}[b,f_k](1-\chi_{B_S})\| \leq \frac{\varepsilon}{5}, \andx \|(1-\chi_{B_S})[b,f_k]\chi_{B_R}\| \leq \frac{\varepsilon}{5}.
\end{equation}
By assumption, we have $\|[b,f_k]\|> \varepsilon$; and since $(f_k)$ is very Lipschitz, we can also assume that $f_k|_{B_S} \approx_{\varepsilon/10} \gamma$ for some constant $\gamma$. Hence, we have:
\begin{equation}\label{commut est 2}
  \|\chi_{B_S}[b,f_k]\chi_{B_S}\| \leq 2 \cdot \frac{\varepsilon}{10}\cdot\|b\| \leq \frac{\varepsilon}{5}.
\end{equation}
Now cutting the space by $B_R, B_R^c$ and $B_S, B_S^c$, we have the following decomposition for the operator $T=[b,f_k]$ (recall we assume that $S>R$):
$$T=(1-\chi_{B_R})T(1-\chi_{B_R})+\chi_{B_R}T\chi_{B_S}+\chi_{B_R}T(1-\chi_{B_S})+(\chi_{B_S}-\chi_{B_R})T\chi_{B_R}+(1-\chi_{B_S})T\chi_{B_R}.$$
From Inequalities (\ref{commut est 1}), the norms of the third and fifth items are less than or equal to $\frac{\varepsilon}{5}$; and from Equality (\ref{commut est 2}), the norms of the second and fourth items are less than or equal to $\frac{\varepsilon}{5}$ as well. Hence by triangle inequality, we have:
\begin{eqnarray*}
  \|(1-\chi_{B_R})[b,f_k](1-\chi_{B_R})\| & \geq & \|[b,f_k]\|-(\|\chi_{B_R}[b,f_k]\chi_{B_S}\| + \|(\chi_{B_S}-\chi_{B_R})[b,f_k]\chi_{B_R}\|) \\
   & & - (\|\chi_{B_R}[b,f_k](1-\chi_{B_S})\| + \|(1-\chi_{B_S})[b,f_k]\chi_{B_R}\|)\\
   &>& \varepsilon- \frac{2\varepsilon}{5} - \frac{2\varepsilon}{5}\\
   &=& \frac{\varepsilon}{5}.
\end{eqnarray*}
In conclusion, for every $R>0$ and $K \in \mathbb N$, there exists $k \geq K$ such that:
\begin{equation}\label{commut est fin}
  \|(1-\chi_{B_R})[b,f_k](1-\chi_{B_R})\|  >  \frac{\varepsilon}{5}.
\end{equation}
As the original proof for $p=2$ \cite[Section 5]{spakula2017relative}, we may choose sequences: $k_1 < k_2 < \ldots$ and $R_1,R_2,\ldots$ satisfying $R_i \geq 6R_{i-1}$ and
$$\|(\chi_{B_{R_i}}-\chi_{B_{6R_{i-1}}})[b,f_{k_i}](\chi_{B_{R_i}}-\chi_{B_{6R_{i-1}}})\| >  \frac{\varepsilon}{5}.$$
Applying Lemma \ref{VL to LH} to $(f_{k_i})$ and $(R_i)$, we obtain $g \in C_h(X)$ defined by:
$$g:=g_{(f_{k_i}),(R_i)} = \sum_{i=1}^{\infty} f_{k_i}(e_{R_i}-e_{3R_{i-1}}).$$
By the choice of $(R_i)$ above, for any $i$, we have
$$(\chi_{B_{R_i}}-\chi_{B_{6R_{i-1}}})[b,g](\chi_{B_{R_i}}-\chi_{B_{6R_{i-1}}})=(\chi_{B_{R_i}}-\chi_{B_{6R_{i-1}}})[b,f_{k_i}](\chi_{B_{R_i}}-\chi_{B_{6R_{i-1}}}).$$
Hence for any $S>0$, choose an $i$ such that $6R_{i-1}>S$, and we have
\begin{eqnarray*}
  \|[b,g](1-\chi_{B_S})\| & \geq & \|(\chi_{B_{R_i}}-\chi_{B_{6R_{i-1}}})[b,g](\chi_{B_{R_i}}-\chi_{B_{6R_{i-1}}})\| \\
  &= & \|(\chi_{B_{R_i}}-\chi_{B_{6R_{i-1}}})[b,f_{k_i}](\chi_{B_{R_i}}-\chi_{B_{6R_{i-1}}})\| \\
  &> & \frac{\varepsilon}{5},
\end{eqnarray*}
which contradicts with the hypothesis that $[b,g] \in \K$.
\end{proof}

\section{Proof of ``(i) $\Leftrightarrow$ (iv)''}
In this section, we will prove the remaining case of ``(i) $\Leftrightarrow$ (iv)" in Theorem~\ref{char for Roe alg Thm}. Recall that as explained in Section 3, ``(iv) $\Rightarrow$ (i)" holds in general. So we will only focus on the opposite implication ``(i) $\Rightarrow$ (iv)".

A key ingredient to prove ``(i) $\Rightarrow$ (iv)'' is to approximate a bounded operator via its block cutdowns as indicated in \cite[Corollary 4.3]{spakula2017relative} for the case of $p=2$.
%In fact, the identical proof of \cite[Corollary 4.3]{spakula2017relative} works for any $p\in (1,\infty)$ but not for $p=1$ due to the lack of reflexivity on $L^1$-spaces. 
Unfortunately, their proof uses the technique of von Nuemann algebra, which is no longer available for $p \neq 2$. Hence we need to search for a substitution of \cite[Corollary 4.3]{spakula2017relative}, and we figure out the following crucial result, which might be of independent interest to experts in Banach space theory.

\begin{prop}\label{estimate for cmmt cor}
Let $(X,d)$ be a proper metric space equipped with a Radon measure $\mu$ whose support is $X$ and $p \in [1,+\infty)$. Suppose $E$ is a Banach space, $a \in \BLpE$ and $a\in \Commut(L,\varepsilon)$ for some $L,\varepsilon>0$. Let $(e_j)_{j\in J}$ be an equicontinuous family of positive contractions in $C_b(X)$ with $2/L$-disjoint supports, and define $e:=\sum_{j\in J} e_j$. Then, we have
$$\big\|eae-\sum_{j\in J}e_jae_j\big\| \leq \varepsilon.$$
\end{prop}

The proof of the above proposition is technical and relatively long, so we decide to postpone it to Section \ref{app} for the convenience of the reader, and first show how to use the proposition to prove ``(i) $\Rightarrow$ (iv)". Let us start with the following lemma, which is a consequence of Proposition~\ref{estimate for cmmt cor} by the same proof of \cite[Lemma~4.6]{spakula2017relative}. It may be worth reminding the reader that for any $L,\varepsilon>0$, we denote $$\Commut(L,\varepsilon)=\big\{a \in \BLpE: \|[a,f]\| < \varepsilon, \mbox{~for~any~}L\mbox{-Lipschitz~} f\in C_b(X)_1\big\}.$$

\begin{lem}\label{decomp for operator lemma}
Let $\Xs$ and $\Ys$ be metric families of $X$ such that $\Xs \xlongrightarrow{4L^{-1}+4} \Ys$ for some $L>0$, and $a\in \BLpE$ be block diagonal with respect to $\Xs$ for $p\in [1,\infty)$. Let $\varepsilon>0$ be such that $a\in \Commut(L,\varepsilon)$. Then we can write:
\begin{equation}\label{EQ1}
a\approx_{8\varepsilon}a_{00}+a_{01}+a_{10}+a_{11},
\end{equation}
where each $a_{ii'}$ is of the form $\theta_{(f_k)_{k\in K}}(gag')$ (see (\ref{defn for theta}))for some $g,g'\in C_b(X)_1$ and some equicontinuous positive family $(f_k)_{k\in K}$ in $C_b(X)_1$ with disjoint supports, such that each $\supp(f_k)$ is contained in some set in $\N_{L^{-1}+1}(\Ys)$. In particular:
\begin{itemize}
  \item [(i)] each $a_{ii'}$ is block diagonal with respect to $\N_{L^{-1}+1}(\Ys)$,
  \item [(ii)] if $a\in \Commut(L',\varepsilon')$ for some $L',\varepsilon'>0$, then each $a_{ii'}$ is in $\Commut(L',\varepsilon')$ as well, and
  \item [(iii)] if $B\subseteq \BLpE$ is a Banach subalgebra such that $C_b(X)BC_b(X)= B$ and $B$ is closed under block cutdowns, and if $a$ is in $B$, then each $a_{ii'}$ is in $B$ as well.
\end{itemize}
\end{lem}

\begin{proof}[Proof of Theorem \ref{char for Roe alg Thm}, ``(i)$\Rightarrow$ (iv)"]
Although the proof is exactly the same as the one given in \cite{spakula2017relative}, we decide to include it here for the completeness and show the reader how straight finite decomposition complexity is used in the proof.

Take $b\in B$ such that it commutes with all $f\in \VLinf$. Given $\varepsilon> 0$, we aim to construct a finite propagation operator in $B$, which is $\varepsilon$-close to $b$. It follows from Lemma~\ref{commut lemma} that for every
$$\varepsilon_n:=\varepsilon/(2\cdot 8^n),$$
there exists some $L_n>0$ such that $b\in Commut(L_n,\varepsilon_n)$. Set
$$R_n:= 4(L_n^{-1}+1)+2(L_{n-1}^{-1}+1)+\cdots +2(L_1^{-1}+1).$$
Since $X$ has straight finite decomposition complexity, there exist metric families
$\Xs_0=\{X\},\Xs_1,\ldots,\Xs_m$ such that $\Xs_{n-1}\xlongrightarrow{R_n}\Xs_n$ for $n\in\{1,\ldots, m\}$ and $\Xs_m$ is uniformly bounded. An elementary observation shows that
\begin{equation}\label{EQ2}
\N_{(L_{n-1}^{-1}+1)+\cdots +(L_1^{-1}+1)}(\Xs_{n-1})\xlongrightarrow{4(L_n^{-1}+1)}\N_{(L_{n-1}^{-1}+1)+\cdots +(L_1^{-1}+1)}(\Xs_{n}).
\end{equation}
Thus, we can apply Lemma \ref{decomp for operator lemma} inductively with $L_n$, $\varepsilon_n$, the operators obtained in the previous iteration, and metric families in (\ref{EQ2}). After $m$ steps, we approximate $b$ by an operator $b'$ which is a sum of $4^m$ operators in $B$, each of which is block diagonal with respect to the bounded family $\N_{(L_{m}^{-1}+1)+\cdots +(L_1^{-1}+1)}(\Xs_m)$. Hence, operators which are block diagonal with respect to it clearly have finite propagation. Consequently, $b'\in\Roe$. Finally, the distance between $b$ and $b'$ is at most
$$8\varepsilon_{1}+4(8\varepsilon_{2}+4(8\varepsilon_{3}+4(\ldots)))=\varepsilon(\frac{1}{2}+\frac{1}{4}+\frac{1}{8}+\ldots)=\varepsilon$$
by Lemma~\ref{decomp for operator lemma}. So we finish the proof.
\end{proof}

\subsection{Approximation via block cutdowns}\label{app} Finally, we complete the proof of Proposition~\ref{estimate for cmmt cor} as promised before.

The main difficulty is the lack of reflexivity of the $L^p$-Bochner space $\LpE$ for general $p$ and general Banach space $E$ (see e.g. \cite{MR1027088, MR0453964, MR0206190}), which impedes us from applying the original proof in \cite{spakula2017relative} directly. Instead, we establish some substituting results in functional analysis and state them in the context of general Banach spaces, which conceivably would be of independent interests.

\emph{In the rest of this subsection, suppose $\X$ is a Banach space and $\Xl$ is a closed subspace of the dual space $\X^*$, which separates points in $\X$} (i.e., for any nonzero $\xi \in \X$, there exists some $\eta \in \Xl$ such that $\eta(\xi) \neq 0$). The inclusion $i: \Xl \hookrightarrow \X^*$ induces a surjective adjoint map $i^*: \X^{**} \rightarrow \Xl^*$. Composing it with the canonical map from $\X$ into its double dual $\X^{**}$, we obtain the following map
\begin{equation}\label{def for tau}
  \tau: \X \rightarrow \Xl^*.
\end{equation}
It is clear that $\tau$ is injective, as $\Xl$ separates points in $\X$.

For any $\theta \in \Xl^*$ and $\eta \in \Xl$, we use the notation $\l \theta,\eta \r$ for $\theta(\eta)$. Consider the Banach space $\BXXl$ of all bounded operators from $\X$ to $\Xl^*$, equipped with the \emph{weak* operator topology (W*OT)}\footnote{In \cite{spakula2017relative} \v{S}pakula and Tikuisis considered the weak operator topology (WOT) instead. However, (WOT) and (W*OT) agree when $\X^* \cong \X$ and taking $\Xl:=\X^*$.} defined as follows: a net $\{T_\alpha\}$ converges to $T$ in $\BXXl$ if and only if for any $\xi \in \X$ and any $\eta \in \Xl$, we have
$$\l T_\alpha(\xi),\eta \r \rightarrow \l T(\xi),\eta\r.$$
%By \cite[Theorem~3.10]{rudin2006functional}, we immediately obtain the following {\color{red} where do you use Lemma 3.2?}:
%\begin{lem}
%Every continuous functional $\phi$ on $\big(\BXXl,\W*OT\big)$ has the following form:
%$$\phi(T)=\sum_{i=1}^{n} \l T(\xi_i),\eta_i \r,$$
%for all $T \in \BXXl$ and for some $\xi_i \in \X$ and some $\eta_i \in \Xl$.
%\end{lem}
The \emph{strong* topology with respect to $\Xl$} on $\BX$ is defined as follows: a net $\{T_\alpha\}$ converges to $T$ in $\BX$ if and only if for any $\xi \in \X$ and any $\eta \in \Xl$, we have
$$\|T_\alpha(\xi)-T(\xi)\|  \rightarrow 0 \andx \|T_\alpha^*(\eta)-T^*(\eta)\| \rightarrow 0.$$

We say that $\Xl$ is \emph{$a^*$-invariant} for $a\in \BX$ if $a^*(\Xl) \subseteq \Xl$. In this case, the restriction $a^*|_{\Xl}$ belongs to $\mathfrak{B}(\Xl)$. Hence, its adjoint $(a^*|_{\Xl})^*$ belongs to $\BXl$ as well. In order to simplify notations, we write $a^{(**)}$ instead of $(a^*|_{\Xl})^*$. Clearly, for any $\zeta \in \Xl^*$ and $\eta \in \Xl$ we have:
$$\l a^{(**)}\zeta, \eta\r=\l \zeta, a^*\eta \r.$$
Moreover, it is easy to check that if $\Xl$ is $a^*$-invariant for some $a\in \BX$, then
\begin{equation}\label{comm diagm}
  a^{(**)}\tau=\tau a.
\end{equation}
In other words, the following diagram commutes:
\begin{displaymath}
\xymatrix@=1.3cm{
  \X \ar@{^{(}->}[r]^-{\textstyle \tau} \ar[d]_{\textstyle a}& \Xl^* \ar[d]^{\textstyle a^{(**)}}  \\
  \X  \ar@{^{(}->}[r]^-{\textstyle \tau} & \Xl^*.    }
\end{displaymath}

We say that $\Xl$ is \emph{$\mathcal{A}^*$-invariant} for a subset $\mathcal{A}\subseteq \BX$ if $\Xl$ is $a^*$-invariant for all $a\in \mathcal{A}$. If $G$ is a subgroup of invertible isometries in $\BX$ and $\hat{\X}$ is $G^*$-invariant, then $u^*(\hat{\X})=\hat{\X}$ for all $u\in G$. It is clear that if $u$ is an invertible isometry, then so are $u^{*}$ and $u^{(**)}$. Moreover, $(u^{*})^{-1}=(u^{-1})^{*}$ and $(u^{(**)})^{-1}=(u^{-1})^{(**)}$, which are denoted by $u^{-*}$ and $u^{-(**)}$, respectively. Now suppose $(u_\alpha)$ and $u$ are invertible isometries in $G$ and $u_\alpha \rightarrow u$ in the strong* topology with respect to $\Xl$, then $\|u_\alpha^{-*}\eta- u^{-*}\eta \| \rightarrow 0$ for all $\eta \in \Xl$. Indeed, we have
\begin{equation}\label{estimate}
\|u_\alpha^{-*}\eta- u^{-*}\eta \|=\|\eta- u_\alpha^*u^{-*}\eta \|=\|u^*(u^{-*}\eta)- u_\alpha^*(u^{-*}\eta)\| \rightarrow 0.
\end{equation}

We have the following technical lemma, which generalises \cite[Lemma~4.1]{spakula2017relative}.

\begin{lem}\label{unique expectation lemma}
Suppose $G$ is an abelian subgroup of the group of invertible isometries in $\BX$, which is compact in the strong* topology with respect to $\Xl$. Suppose $\Xl$ is $G^*$-invariant, and define $\G':= \{a\in \BXXl: au=u^{(**)}a, \forall u \in G \}$.

Then there exists a unique idempotent linear contraction $\E_G: \BXXl \rightarrow \G'$ with the following properties:
\begin{enumerate}[1)]
  \item For any $b_1,b_2 \in G$ and $a \in \BXXl$, $\E_G(b_1^{(**)}ab_2)=b_1^{(**)}\E_G(a)b_2$.
  \item The restriction of $\E_G$ to the unit ball of $\BXXl$ is  (W*OT)-continuous.
\end{enumerate}
In this case, for any $a \in \BXXl$, we have that
\begin{equation}\label{norm estimamte for cmmt}
\|\E_G(a)-a\| \leq \sup_{u \in G} \|au-u^{(**)}a\|.
\end{equation}

\end{lem}

\begin{proof}
Since $G$ is compact with respective to the strong* topology, we consider the normalised Haar measure $\mu_G$ on $G$. Fix $a \in \BXXl$, the map
$$(G,\mbox{strong* topology}) \rightarrow (\BXXl, \W*OT)$$
defined by $u \mapsto u^{-(**)}au$ is clearly continuous. For each $\xi \in \X$ and each $a \in  \BXXl$, we may consider the following functional on $\Xl$:
$$\phi_{\xi,a}: \eta \mapsto \int_G\l u^{-(**)}au\xi, \eta \r \d \mu_G(u),$$
whose norm is bounded by $\|a\|\cdot \|\xi\|$. Therefore, we obtain a linear contraction $\E_G: \BXXl \rightarrow \BXXl$ given by $\E_G(a)(\xi)=\phi_{\xi,a}$, where $\xi\in \X$ and $a \in  \BXXl$.

 %Clearly, for any $\xi \in \X$ and $\eta \in \Xl$, we have
%$$\l \E_G(a)\xi, \eta \r = \int_G\l u^{-(**)}au\xi, \eta \r \d \mu_G(u).$$
%Hence we may write $\E_G(a)$ in the form of Pettis integration:
%$$\E_G(a)={\scriptstyle\W*OT-}\int_G u^{-(**)}au \d \mu_G(u).$$

It remains to check that $\E_G$ satisfies the required properties. First of all, we show that $\E_G$ has image in $\G'$. More precisely, $\E_G(a)v=v^{(**)}\E_G(a)$ for any $a \in \BXXl$ and any $v \in G$. Given $\xi \in \X$ and $\eta \in \Xl$, it follows from the right-invariance of the Haar measure $\mu_G$ that
\begin{eqnarray*}
  \l \E_G(a)v\xi, \eta \r &=& \int_G\l u^{-(**)}auv\xi, \eta \r \d \mu_G(u) \\
   &=& \int_G\l v^{(**)}u^{-(**)}au\xi, \eta \r \d \mu_G(u) \\
   &=& \int_G\l u^{-(**)}au\xi, v^*\eta \r \d \mu_G(u)\\
   &=& \l \E_G(a)\xi, v^*\eta \r\\
   &=& \l v^{(**)}\E_G(a)\xi, \eta \r.
\end{eqnarray*}
Hence, it follows that $\E_G(a)v=v^{(**)}\E_G(a)$.

Given $a \in \BXXl$, $\xi \in \X$ and $\eta \in \Xl$, we have
\begin{eqnarray*}
  |\l (\E_G(a)-a)\xi,\eta \r| & \leq & \int_G  \|u^{-(**)}au-a\| \cdot \|\xi\| \cdot \|\eta\| \d \mu_G(u)\\
   & = & \int_G  \|au-u^{(**)}a\| \cdot \|\xi\| \cdot \|\eta\| \d \mu_G(u)\\
   & \leq & \left(\sup_{u \in G} \|au-u^{(**)}a\| \right) \cdot \|\xi\| \cdot \|\eta\|.
\end{eqnarray*}
Hence, (\ref{norm estimamte for cmmt}) holds. In particular, $\E_G(a)=a$ for any $a \in \G'$, which implies that $\E_G: \BXXl \rightarrow \G'$ is an idempotent.

Now let us check that $\E_G(b_1^{(**)}ab_2)=b_1^{(**)}\E_G(a)b_2$ for any $b_1,b_2 \in G$ and $a \in \BXXl$. Since $G$ is abelian, we have
\begin{eqnarray*}
  \l \E_G(b_1^{(**)}ab_2)\xi,\eta \r &=& \int_G \l u^{-(**)}b_1^{(**)}ab_2u\xi, \eta \r \d \mu_G(u) \\
   &=& \int_G \l b_1^{(**)}u^{-(**)}aub_2\xi, \eta \r \d \mu_G(u) \\
   &=& \int_G \l u^{-(**)}aub_2\xi, b_1^{*}\eta \r \d \mu_G(u)\\
   &=& \l \E_G(a)b_2\xi,b_1^{*}\eta \r \\
   &=& \l b_1^{(**)}\E_G(a)b_2\xi,\eta \r,
\end{eqnarray*}
for any $\xi \in \X$ and any $\eta \in \Xl$. Hence, $\E_G(b_1^{(**)}ab_2)=b_1^{(**)}\E_G(a)b_2$.

In order to prove the (W*OT)-continuity of the restriction of $\E_G$ to the unit ball of $\BXXl$, we have to approximate the integration by finite Riemann sums \emph{uniformly} in the weak* operator topology:

Indeed, fix $\xi \in \X$, $\eta \in \Xl$ and $u \in G$ and for any $\varepsilon>0$, from (\ref{estimate}) there exists an open neighbourhood $V_u$ of $u$ in the strong* topology such that for all $v \in V_u$ and all $a \in \BXXl_1$, we have
$$|\l v^{-(**)}av\xi, \eta \r - \l u^{-(**)}au\xi, \eta \r| <\varepsilon.$$

Since $\{V_u: u \in G\}$ forms an open cover of $G$ and $G$ is compact in the strong* topology, there exists a finite subcover $\{V_{u_1},\ldots, V_{u_n}\}$ of $G$. Let $W_1=V_{u_1}$ and we put
$W_k=V_{u_k}\setminus \bigcup_{i=1}^{k-1} W_i$
for $1<k\leq n$. Without loss of generality, we may assume that $\{W_k\}_{k=1}^n$ forms a non-empty Borel partition of $G$. Take an arbitrary point $w_k$ in each $W_k$ for $k=1,\ldots,n$. Then for any $a \in \BXXl_1$ and $u \in W_k$, we have that
$$|\l u^{-(**)}au\xi, \eta \r - \l w_k^{-(**)}aw_k\xi, \eta \r| <2\varepsilon.$$ In particular, we have that
\begin{eqnarray*}
&  &\big| \l \E_G(a) \xi, \eta \r - \sum_{k=1}^n \l w_k^{-(**)} a w_k \xi, \eta \r \mu_G(W_k)\big| \\
& = & \big| \sum_{k=1}^n \int_{W_k} \l u^{-(**)}a u\xi, \eta \r \d \mu_G(u) - \sum_{k=1}^n \int_{W_k} \l w_k^{-(**)}a w_k \xi, \eta \r \d \mu_G(u) \big| \\
& \leq & \sum_{k=1}^n \int_{W_k}\big| \l u^{-(**)}a u\xi, \eta \r - \l w_k^{-(**)}a w_k \xi, \eta \r  \big| \d \mu_G(u)\\
& \leq &  \sum_{k=1}^n \int_{W_k} 2\varepsilon \d \mu_G(u) = 2\varepsilon,
\end{eqnarray*}
for all $a\in \BXXl_1$. Since the map $a\mapsto \sum_{k=1}^n\mu_G(W_k)w_k^{-(**)}aw_k$ is continuous in the weak* operator topology, it is not hard to see that the restriction of $\E_G$ to the unit ball of $\BXXl$ is (W*OT)-continuous as well.

Finally, we check the uniqueness of $\E_G$. If we have another $\E: \BXXl \rightarrow \G'$ satisfying all the conditions in the lemma, then:
\begin{align*}
\E_G(a) &= \E(\E_G(a))            && (\E \mbox{~fixes~} \G' )  \\
            &= \E\big({\scriptstyle\W*OT-}\int_G u^{-(**)}au \d \mu_G(u)\big)     && \\
            &= {\scriptstyle\W*OT}-\int_G \E(u^{-(**)}au) \d \mu_G(u)        &&(\W*OT\mbox{-continuity on the unit ball} ) \\
            &= {\scriptstyle\W*OT-}\int_G u^{-(**)}\E(a)u \d \mu_G(u)    &&(\mbox{Property~1)}) \\
            &= \E_G(\E(a))         && \\
            &= \E(a)         && (\E_G \mbox{~fixes~} \G' )
\end{align*}
for all $a\in \BXXl_1$. Thus, $\E_G=\E$ and we complete the proof.
\end{proof}

Now let us return to the setting of Proposition~\ref{estimate for cmmt cor}. Let $(X,d)$ be a proper metric space equipped with a Radon measure $\mu$ whose support is $X$. Let $q$ be the conjugate exponent to $p$ when $p\in (1,+\infty)$, and $q=0$ when $p=1$. Suppose $E$ is a Banach space and $(e_j)_{j\in J}$ is an equicontinuous family of positive contractions in $C_b(X)$ with uniformly disjoint supports.

In order to apply Lemma~\ref{unique expectation lemma}, we put $\X=\LpE$ and $\Xl=\LqEst$. Clearly, $\Xl$ is a closed subspace of the dual space $\X^*$, and separates points in $\X$ by Lemma~\ref{norm isom}. For each $j\in J$, set $A_j= \supp(e_j)$ and $B=X\setminus \big(\bigsqcup_{j\in J} A_j\big)$. We consider $p_j$ and $q_c$ in $\BLpE$ given by $p_j(\xi)=\chi_{A_j}\xi$ and $q_c(\xi)=\chi_{B}\xi$ for $\xi \in \LpE$. We define that
\begin{equation}\label{G}
G=\left\{\sum_{j\in J} (-1)^{\alpha_j} p_j + (-1)^\beta q_c: (\alpha_j)_{j\in J} \subseteq (\mathbb Z/2)^J, \beta \in \mathbb Z/2  \right\},
\end{equation}
where the sum converges in (SOT) and each element in $G$ can be presented by a function of the form $\sum_{j\in J} (-1)^{\alpha_j} \chi_{A_j} + (-1)^\beta \chi_{B}$ (in the pointwise convergence) via the faithful multiplication representation $\rho:L^\infty(X) \rightarrow \BLpE$.

Since $g^2=\text{id}$ for all $g\in G$, $G$ becomes a subgroup of the invertible isometry group in $\BLpE$, and clearly $G$ is abelian. Also notice that $\Xl$ is $L^\infty(X)^*$-invariant as for any $f\in L^\infty(X)\subseteq \BLpE$ and $\eta \in \Xl$, we have that $f^*(\eta)=f\cdot\eta$ by pointwise multiplications as functions on $X$.\footnote{It is worth noting that $C_0(X,E^*)$ is not $L^\infty(X)^*$-invariant and this is the reason why we use $L^0(X;E^*)$ instead of $C_0(X,E^*)$ when $p=1$.} Consequently, $\Xl$ is $G^*$-invariant since $G\subseteq \rho(L^\infty(X))$. Moreover, the strong* topology on $G$ with respect to $\Xl$ is compact, as it is homeomorphic to the product topology on $(\mathbb Z/2)^{J\cup \{\beta\}}$.\footnote{However, it is false for $L^\infty(X;E^*)$ and this is the reason why we use $L^0(X;E^*)$ instead of $L^\infty(X;E^*)$ when $p=1$.}

The next lemma is a replacement of \cite[Corollary~4.2]{spakula2017relative}, where \v{S}pakula and Tikuisis work within the setting of von Neumann algebras. Instead, we provide a direct and concrete proof here as follows:

\begin{lem}\label{construction}
As above, the group $G$ is defined as in (\ref{G}) and $q$ is the conjugate exponent to $p$ when $p\in (1,\infty)$, and $q=0$ when $p=1$. Let $\X=\LpE$ and $\Xl=\LqEst$. If $\G'=\{a\in \BXXl: au=u^{(**)}a, \forall u \in G \}$, then there exists a (W*OT)-continuous idempotent linear contraction  $\E: \BXXl \rightarrow \G'$ given by the formula
$$\E(x)=\sum_{j\in J} p_j^{(**)}xp_j + q_c^{(**)}xq_c,$$
where the sum converges in (SOT). Moreover, $\E(b_1^{(**)}ab_2)=b_1^{(**)}\E(a)b_2$ for any $b_1,b_2 \in G$ and $a \in \BXXl$. Consequently, we have that
$$\|\E(a)-a\| \leq \sup_{u \in G} \|au-u^{(**)}a\|, \quad \text{for any $a \in \BXXl$.}$$

\end{lem}

\begin{proof}
It is clear that $\E$ is a (W*OT)-continuous linear map on $\BXXl$ and the sum defining $\E$ converges in (SOT), so we leave the details to the readers.

Let us first verify that $\E$ is a contraction. When $p=1$, we have
\begin{eqnarray*}
\|\E(x)\xi\|    \leq  \sum_{j\in J} \|xp_j\xi\| + \|xq_c\xi\| \nonumber
 \leq  \|x\| \cdot \big(\sum_{j\in J} \|\chi_{A_j}\xi\|_1 + \|\chi_{B}\xi\|_1 \big) = \|x\| \cdot \|\xi\|_1,
\end{eqnarray*}
for any $\xi \in \LE$ by Lemma~\ref{norm isom}. It implies that $\E$ is a contraction in this case. When $p>1$, it follows from H\"{o}lder's inequality that
\begin{eqnarray*}
\big| \big\l  \sum_{j\in J} p_j^{(**)}xp_j\xi + q_c^{(**)}xq_c\xi, \eta \big\r  \big|
&\leq&  \sum_{j\in J} |\l xp_j\xi, p_j^*\eta \r| + |\l xq_c\xi, q_c^*\eta \r| \\
& \leq&  \|x\| \cdot \big(\sum_{j\in J} \|p_j\xi\|_p \cdot \|p_j^*\eta\|_q + \|q_c\xi\|_p \cdot\|q_c^*\eta\|_q \big) \\
& \leq & \|x\| \cdot \big(\sum_{j\in J} \|p_j\xi\|_p^p + \|q_c\xi\|_p^p\big)^{\frac{1}{p}} \cdot \big(\sum_{j\in J} \|p_j^*\eta\|_q^q + \|q_c^*\eta\|_q^q\big)^{\frac{1}{q}} \\
&  = &  \|x\| \cdot \|\xi\|_p \cdot \|\eta\|_q,
\end{eqnarray*}
for any $\xi \in \LpE$ and $\eta \in \LqEst$. This implies that
$$\|\E(x)\xi\|=\big\| \sum_{j\in J} p_j^{(**)}xp_j\xi + q_c^{(**)}xq_c\xi \big\| \leq \|x\| \cdot \|\xi\|_p$$
by Lemma \ref{norm isom}. Hence, $\E$ is a contraction in this case as well.

Now we show that the image of $\E$ sits inside $\G'$. Indeed, given any $x\in \BXXl$, any $u=\sum_{j\in J} (-1)^{\alpha_j} p_j + (-1)^\beta q_c \in G$, $\xi \in \X$ and $\eta \in \Xl$, we have that
\begin{eqnarray*}
\l \E(x)u\xi, \eta \r &=& \big\l \big( \sum_{j\in J} p_j^{(**)}xp_j + q_c^{(**)} x q_c\big) \big(\sum_{j\in J}(-1)^{\alpha_j} p_j +(-1)^\beta q_c \big)\xi, \eta \big\r\\
         &=& \sum_{j\in J} (-1)^{\alpha_j} \l p_j^{(**)}x p_j\xi,\eta \r + (-1)^\beta \l q_c^{(**)}xq_c\xi, \eta \r\\
         &=& \sum_{j\in J} (-1)^{\alpha_j} \l x p_j\xi, p_j^*\eta \r + (-1)^\beta \l xq_c\xi, q_c^*\eta \r.
\end{eqnarray*}
On the other hand,
\begin{eqnarray*}
\l u^{(**)}\E(x)\xi, \eta \r &=& \big\l \big( \sum_{j\in J} p_j^{(**)}xp_j + q_c^{(**)} x q_c\big)\xi, u^*\eta \big\r\\
         &=& \sum_{j\in J} \l x p_j\xi, p_j^*u^*\eta \r + \l xq_c\xi, q_c^*u^*\eta \r\\
         &=& \sum_{j\in J} (-1)^{\alpha_j} \l x p_j\xi, p_j^*\eta \r + (-1)^\beta \l xq_c\xi, q_c^*\eta \r.
\end{eqnarray*}
Hence, $\E(x)u=u^{(**)}\E(x)$ for all $u\in G$.

Next, we show that $\E(x)=x$ for all $x\in \G'$. In other words, $\E$ is an idempotent onto $\G'$. Fix an $x \in \G'$ and for any $u=\sum_{j\in J} (-1)^{\alpha_j} p_j + (-1)^\beta q_c$ in $G$, we have that
$$
p_j^{(**)}xp_i=(-1)^{\alpha_i}p_j^{(**)}x(up_i)=(-1)^{\alpha_i}(p_j^{(**)}u^{(**)})xp_i=(-1)^{\alpha_i+\alpha_j}p_j^{(**)}xp_i.
$$
It follows that $p_{j}^{(**)}xp_i=0$ for any $i\neq j$. Similarly, $p_j^{(**)}xq_c=q_c^{(**)}xp_{j}=0$ for any $j \in J$. Therefore, we have that
\begin{eqnarray*}
x =  \big( \sum_{i\in J}p_i +q_c \big)^{(**)}~x~\big( \sum_{j\in J}p_j +q_c \big)=\sum_{j\in J} p_j^{(**)}xp_j + q_c^{(**)}xq_c = \E(x)\  \text{for all $x\in \G'$.}
\end{eqnarray*}

Moreover, for any $b_1,b_2 \in G$ and any $x\in \BXXl$ we have that
\begin{eqnarray*}
  \E(b_1^{(**)}xb_2) &=& \sum_{j\in J} p_j^{(**)}b_1^{(**)}xb_2p_j + q_c^{(**)}b_1^{(**)}xb_2q_c\\
                  &=& \sum_{j\in J} b_1^{(**)}p_j^{(**)}xp_jb_2 + b_1^{(**)}q_c^{(**)}xq_cb_2\\ &=&b_1^{(**)}\E(x)b_2,
\end{eqnarray*}
where we use the fact that $b_kp_j=p_jb_k$ and
$b_kq_c=q_cb_k$ for any $j\in J$ and $k\in \{1,2\}$.

The final conclusion follows from the uniqueness of $\E$ in Lemma~\ref{unique expectation lemma} and \eqref{norm estimamte for cmmt} therein. So we finish the proof.
\end{proof}

\begin{proof}[Proof of Proposition \ref{estimate for cmmt cor}]
Let the group $G$ be defined as in (\ref{G}), and the map $\E: \Bd \rightarrow \G'$ be the idempotent defined in Lemma \ref{construction}. Recall that by Lemma~\ref{norm isom}, the map $\tau: \LpE \rightarrow \LqEst^*$ defined in (\ref{def for tau}) is an isometric embedding, hence it induces the following isometric embedding
$$\iota: \BLpE=\mathfrak{B}(\LpE,\LpE) \hookrightarrow \Bd.$$
In other words, $\iota(a)=\tau \circ a$ for any $a \in \BLpE$.

Now we define another map $\E': \BLpE \rightarrow \BLpE$ by the formula
$$\E'(z)=\sum_{j\in J} p_jzp_j + q_czq_c$$
for $z \in \BLpE$ and the sum converges in (SOT) by Lemma~\ref{SOT convergence lem}. It follows easily from Equation~(\ref{comm diagm}) that the following diagram commutes
\begin{displaymath}
\xymatrix@=1.5cm{
  \BLpE \ar@{^{(}->}[r]^-{\textstyle \iota} & \Bd  \\
  \BLpE \ar[u]^{\textstyle \E'} \ar@{^{(}->}[r]^-{\textstyle \iota} & \Bd.  \ar[u]_{\textstyle \E}  }
\end{displaymath}
Furthermore, we have that
$$\|\E'(z)-z\|=\|\iota(\E'(z))-\iota(z)\|=\|\E(\iota(z))-\iota(z)\| \leq \sup_{u \in G}\{\|\iota(z)u-u^{(**)}\iota(z)\|\},$$
for any $z \in \BLpE$.
While for $u \in G$, it follows from Equation~(\ref{comm diagm}) that
$$\iota(z)u-u^{(**)}\iota(z)=\tau z u-u^{(**)} \tau z=\tau z u-\tau u z.$$
Combining the above facts together, we obtain that
$$\|\E'(z)-z\| \leq \sup_{u \in G}\{\|zu-uz\|\}.$$

Let $e:=\sum_{j\in J} e_j$. Since $p_je=e_j=ep_j$ and $q_ce=eq_c=0$, we have that
\begin{eqnarray*}
  \E'(eae) &=& \sum_{j\in J} p_jeaep_j + q_ceaeq_c
   = \sum_{j\in J} e_jae_j.
\end{eqnarray*}
Also notice that for any $u=\sum_{j\in J} (-1)^{\alpha_j} p_j + (-1)^\beta q_c$ in $G$, we have that $eu=ue= \sum_{j \in J} (-1)^{\alpha_j} e_j.$ Since $\{A_j\}_{j \in J}$ are pairwise $2/L$-disjoint, there exists an $L$-Lipschitz map $f \in C_b(X)_1$ such that $f|_{A_j}\equiv (-1)^{\alpha_j}\chi_{A_j}$ for all $j$. Hence, $\|[a,f]\| \leq \varepsilon$ since $a\in \Commut(L,\varepsilon)$, and we clearly have $e_jf=fe_j=(-1)^{\alpha_j}e_j$. Therefore, we obtain that
\begin{align*}
  ueae &= \big( \sum_{j \in J} (-1)^{\alpha_j} e_j \big) ae
       = efae
       \approx_{\varepsilon}  eafe
       = ea\big( \sum_{j \in J} (-1)^{\alpha_j} e_j \big)= eaeu.
\end{align*}
Finally, we complete the proof by the following computation:
$$ \|eae-\sum_{j\in J}e_jae_j\| = \|\E'(eae)-eae\| \leq \sup_{u \in G} \|eaeu-ueae\| \leq \varepsilon,$$
for any $a\in \Commut(L,\varepsilon)$.
\end{proof}

{\bf Acknowledgments}. The first-named author would like to thank Tomasz Kania for helpful discussions on Banach space valued $L^p$-spaces.

%{\color{red} We could consider definition of $L^p$ type Roe algebras.}

\bibliographystyle{plain}
\bibliography{bibRELCMT}

\begin{thebibliography}{10}

\bibitem{ALLW17}
Pere Ara, Kang Li, Fernando Lled\'o, and Jianchao Wu.
\newblock Amenability and uniform {R}oe algebras.
\newblock {\em J. Math. Anal. Appl.}, 459(2):686--716, 2018.

\bibitem{MR1027088}
Bahattin Cengiz.
\newblock On the duals of {L}ebesgue-{B}ochner {$L^p$} spaces.
\newblock {\em Proc. Amer. Math. Soc.}, 114(4):923--926, 1992.

\bibitem{chung2018rigidity}
Yeong~Chyuan Chung and Kang Li.
\newblock Rigidity of $\ell^p$ {$R$}oe-type algebras.
\newblock {\em Bulletin of the London Mathematical Society}, 50(6):1056--1070,
  2018.

\bibitem{DF}
Andreas Defant and Klaus Floret.
\newblock {\em Tensor Norms and Operator Ideals}.
\newblock Number 176 in North-Holland Mathematics Studies. North-Holland
  Publishing Co., 1993.

\bibitem{MR0453964}
J.~Diestel and J.~J. Uhl, Jr.
\newblock {\em Vector measures}.
\newblock American Mathematical Society, Providence, R.I., 1977.
\newblock With a foreword by B. J. Pettis, Mathematical Surveys, No. 15.

\bibitem{MR0206190}
N.~Dinculeanu.
\newblock {\em Vector measures}.
\newblock International Series of Monographs in Pure and Applied Mathematics,
  Vol. 95. Pergamon Press, Oxford-New York-Toronto, Ont.; VEB Deutscher Verlag
  der Wissenschaften, Berlin, 1967.

\bibitem{dranishnikov2014asymptotic}
Alexander Dranishnikov and Michael Zarichnyi.
\newblock Asymptotic dimension, decomposition complexity, and {H}aver's
  property {C}.
\newblock {\em Topology Appl.}, 169:99--107, 2014.

\bibitem{Engel18}
Alexander Engel.
\newblock Index theory of uniform pseudodifferential operators.
\newblock {\em preprint}, 2018.
\newblock arXiv:1502.00494.

\bibitem{Engel15}
Alexander Engel.
\newblock Rough index theory on spaces of polynomial growth and
  contractibility.
\newblock {\em preprint}, 2018.
\newblock arXiv:1505.03988.

\bibitem{EwertMeyer}
Eske~Ellen Ewert and Ralf Meyer.
\newblock Coarse geometry and topological phases.
\newblock {\em preprint}, 2018.
\newblock arXiv:1802.05579.

\bibitem{MR1876896}
Erik Guentner and Jerome Kaminker.
\newblock Exactness and the {N}ovikov conjecture.
\newblock {\em Topology}, 41(2):411--418, 2002.

\bibitem{guentner2012notion}
Erik Guentner, Romain Tessera, and Guoliang Yu.
\newblock A notion of geometric complexity and its application to topological
  rigidity.
\newblock {\em Inventiones mathematicae}, 189(2):315--357, 2012.

\bibitem{MR3054574}
Erik Guentner, Romain Tessera, and Guoliang Yu.
\newblock Discrete groups with finite decomposition complexity.
\newblock {\em Groups Geom. Dyn.}, 7(2):377--402, 2013.

\bibitem{MR3451966}
Raffael Hagger, Marko Lindner, and Markus Seidel.
\newblock Essential pseudospectra and essential norms of band-dominated
  operators.
\newblock {\em J. Math. Anal. Appl.}, 437(1):255--291, 2016.

\bibitem{MR1739727}
Nigel Higson and John Roe.
\newblock Amenable group actions and the {N}ovikov conjecture.
\newblock {\em J. Reine Angew. Math.}, 519:143--153, 2000.

\bibitem{MR3158244}
Julian Kellerhals, Nicolas Monod, and Mikael R{\o}rdam.
\newblock Non-supramenable groups acting on locally compact spaces.
\newblock {\em Doc. Math.}, 18:1597--1626, 2013.

\bibitem{Kub}
Yosuke Kubota.
\newblock Controlled topological phases and bulk-edge correspondence.
\newblock {\em Commun. Math. Phys.}, 349(2):493--525, 2017.

\bibitem{MR790430}
B.~V. Lange and V.~S. Rabinovich.
\newblock Noethericity of multidimensional discrete convolution operators.
\newblock {\em Mat. Zametki}, 37(3):407--421, 462, 1985.

\bibitem{LL}
Kang Li and Hung-Chang Liao.
\newblock Classification of uniform {R}oe algebras of locally finite groups.
\newblock {\em J. Operator Theory}, 80(1):25--46, 2018.

\bibitem{LW18}
Kang Li and Rufus Willett.
\newblock Low-dimensional properties of uniform {R}oe algebras.
\newblock {\em J. London Math. Soc.}, (2) 97:98–124, 2018.

\bibitem{MR3212726}
Marko Lindner and Markus Seidel.
\newblock An affirmative answer to a core issue on limit operators.
\newblock {\em J. Funct. Anal.}, 267(3):901--917, 2014.

\bibitem{michor2006functors}
Peter~W Michor.
\newblock {\em Functors and categories of Banach spaces: tensor products,
  operator ideals and functors on categories of Banach spaces}, volume 651.
\newblock Springer, 2006.

\bibitem{MR1763912}
Narutaka Ozawa.
\newblock Amenable actions and exactness for discrete groups.
\newblock {\em C. R. Acad. Sci. Paris S\'er. I Math.}, 330(8):691--695, 2000.

\bibitem{Phil12}
N.~Christopher Phillips.
\newblock Analogs of {C}untz algebras on {$L^p$} spaces.
\newblock {\em preprint, arXiv:1309.4196}, 2012.

\bibitem{phillips2013crossed}
N.~Christopher Phillips.
\newblock Crossed products of {$L^p$} operator algebras and the {$K$}-theory of
  {$C$}untz algebras on {$L^p$} spaces.
\newblock {\em preprint, arXiv:1309.6406}, 2013.

\bibitem{MR918459}
John Roe.
\newblock An index theorem on open manifolds. {I}, {II}.
\newblock {\em J. Differential Geom.}, 27(1):87--113, 115--136, 1988.

\bibitem{MR1399087}
John Roe.
\newblock {\em Index theory, coarse geometry, and topology of manifolds},
  volume~90 of {\em CBMS Regional Conference Series in Mathematics}.
\newblock Published for the Conference Board of the Mathematical Sciences,
  Washington, DC; by the American Mathematical Society, Providence, RI, 1996.

\bibitem{roe2003lectures}
John Roe.
\newblock {\em Lectures on coarse geometry}, volume~31 of {\em University
  Lecture Series}.
\newblock American Mathematical Society, Providence, RI, 2003.

\bibitem{MR2873171}
Mikael R{\o}rdam and Adam Sierakowski.
\newblock Purely infinite {$C^*$}-algebras arising from crossed products.
\newblock {\em Ergodic Theory Dynam. Systems}, 32(1):273--293, 2012.

\bibitem{Scarparo:2016kl}
Eduardo Scarparo.
\newblock Characterizations of locally finite actions of groups on sets.
\newblock {\em Glasg. Math. J.}, 60(2):285--288, 2018.

\bibitem{MR3151282}
Markus Seidel.
\newblock Fredholm theory for band-dominated and related operators: a survey.
\newblock {\em Linear Algebra Appl.}, 445:373--394, 2014.

\bibitem{MR1905840}
Georges Skandalis, Jean-Louis Tu, and Guoliang Yu.
\newblock The coarse {B}aum-{C}onnes conjecture and groupoids.
\newblock {\em Topology}, 41(4):807--834, 2002.

\bibitem{MR2523336}
J{\'a}n {\v{S}}pakula.
\newblock Uniform {$K$}-homology theory.
\newblock {\em J. Funct. Anal.}, 257(1):88--121, 2009.

\bibitem{spakula2017relative}
J{\'a}n {\v{S}}pakula and Aaron Tikuisis.
\newblock Relative commutant pictures of {$R$}oe algebras.
\newblock {\em preprint, arXiv:1707.04552}, 2017.

\bibitem{vspakula2017metric}
J{\'a}n {\v{S}}pakula and Rufus Willett.
\newblock A metric approach to limit operators.
\newblock {\em Trans. Amer. Math. Soc.}, 369(1):263--308, 2017.

\bibitem{vspakula2018quasi}
J{\'a}n {\v{S}}pakula and Jiawen Zhang.
\newblock Quasi-locality and property a.
\newblock {\em arXiv preprint arXiv:1809.00532}, 2018.

\bibitem{MR2800923}
ShuYun Wei.
\newblock On the quasidiagonality of {R}oe algebras.
\newblock {\em Sci. China Math.}, 54(5):1011--1018, 2011.

\bibitem{WW}
Stuart White and Rufus Willett.
\newblock Cartan subalgebras in uniform {R}oe algebras.
\newblock {\em preprint}, 2017.
\newblock arXiv:1808.04410.

\bibitem{WZ10}
Wilhelm Winter and Joachim Zacharias.
\newblock The nuclear dimension of {$C^\ast$}-algebras.
\newblock {\em Adv. Math.}, 224(2):461--498, 2010.

\bibitem{Yu95}
Guoliang Yu.
\newblock Coarse {B}aum-{C}onnes conjecture.
\newblock {\em K-Theory}, 9:199--221, 1995.

\bibitem{MR1451759}
Guoliang Yu.
\newblock Localization algebras and the coarse {B}aum-{C}onnes conjecture.
\newblock {\em $K$-Theory}, 11(4):307--318, 1997.

\bibitem{MR1626745}
Guoliang Yu.
\newblock The {N}ovikov conjecture for groups with finite asymptotic dimension.
\newblock {\em Ann. of Math. (2)}, 147(2):325--355, 1998.

\bibitem{Yu00}
Guoliang Yu.
\newblock The coarse {B}aum-{C}onnes conjecture for spaces which admit a
  uniform embedding into {H}ilbert space.
\newblock {\em Invent. Math.}, 139(1):201--240, 2000.

\bibitem{Zhang18}
Jiawen Zhang.
\newblock Extreme cases of limit operator theory on metric spaces.
\newblock {\em Integral Equations and Operator Theory}, 90(6):73, 2018.

\end{thebibliography}

\end{document}